\documentclass[12pt]{article}
\usepackage{amsthm}
\usepackage{varioref}
\usepackage{amsmath,amstext,amsthm,a4,amssymb,amscd}   
\usepackage[T1]{fontenc}
\usepackage[utf8]{inputenc}
\usepackage{lmodern}
\usepackage{array}
\usepackage{latexsym}
\usepackage{fancyhdr}
\usepackage{mathrsfs}\let\cal\mathscr
\usepackage[all]{xy}
\usepackage{makeidx}   
\usepackage{color}
\usepackage{pifont}
\usepackage{amsfonts,amssymb}
\usepackage{hyperref}
\usepackage{cleveref}
\usepackage[numbers,square]{natbib}
\usepackage{todonotes}

\usepackage{avant}
\usepackage[T1]{fontenc}      

\usepackage{amsfonts,amssymb}  
 
\usepackage{graphicx}
\hypersetup{
    colorlinks=true,
    linkcolor=red,
    filecolor=darkgreen,      
    urlcolor=cyan,
}

\usepackage{natbib}
\newcommand \Om {\Omega}
\newcommand \om {\omega}
\newcommand \0 {\emptyset}

\renewcommand \leq {\leqslant}
\renewcommand \geq {\geqslant}

\newcommand{\norm}[1]{\left\Vert #1\right\Vert}

\DeclareMathOperator{\Vol}{Vol}
\DeclareMathOperator{\End}{End}

\DeclareMathOperator{\Tr}{Tr}

\DeclareMathOperator{\Ker}{Ker}

\DeclareMathOperator{\GL}{GL}
\DeclareMathOperator{\Spec}{Spec}

\DeclareMathOperator{\hol}{hol}

\DeclareMathOperator{\Lie}{Lie}

\newcommand \dbar {\overline{\partial}}

\newcommand \< {\mathcal{h}}
\renewcommand \> {\mathcal{i}}

\newcommand \cinf {\CC^\infty}

\newcommand \Id {{\rm Id}}

\renewcommand \epsilon {\varepsilon}

\newcommand \CC {{\cal C}}

\newcommand \EE {{\cal E}}
\newcommand \FF {{\cal F}}

\newcommand \HH {{\cal H}}

\newcommand \PP {{\cal P}}

\newcommand \UU {{\cal U}}

\newcommand \Q {{\mathcal Q}}

\newcommand \QQ[1] {\Q_{#1,x_0}}

\newcommand \dt {\frac{\partial}{\partial t}}

\newcommand{\til}[1]{\widetilde{#1}}

\newcommand \R {\mathbb R}

\newcommand \C {\mathbb C}
\newcommand \IH {\mathbb H}
\newcommand \N {\mathbb N}

\newcommand \Z {\mathbb Z}

\newcommand \IT {\mathbb T}

\newcommand \fl {\rightarrow}

\newcommand \ignore[1] {}
\newcommand{\overbar}[1]{\mkern 1.5mu\overline{\mkern-1.5mu#1\mkern-1.5mu}\mkern 1.5mu}

\theoremstyle{plain}
\newtheorem{theorem}{Theorem}[section]
\newtheorem{lem}[theorem]{Lemma}
\newtheorem{cor}[theorem]{Corollary}
\newtheorem{prop}[theorem]{Proposition}
\theoremstyle{definition}
\newtheorem*{ackn*}{Acknowledgements}
\newtheorem{defi}[theorem]{Definition}

\newtheorem{rem}[theorem]{Remark}

\numberwithin{equation}{section}

\crefname{equation}{}{}
\crefname{lem}{Lemma}{Lemmas}
\crefname{theorem}{Theorem}{Theorems}
\crefname{cor}{Corollary}{Corollaries}
\crefname{ex}{Example}{Examples}
\crefname{defi}{Definition}{Definitions}
\crefname{prop}{Proposition}{Propositions}
\crefname{section}{Section}{Sections}
\crefname{subsection}{Section}{Sections}
\crefname{rem}{Remark}{Remarks}

\begin{document}

\title{\bf{Quantization and isotropic submanifolds}}
\author{Louis IOOS$^1$}
\date{}

\maketitle

\footnotetext[1]{Supported by the grant DIM-RDF from Région
Ile-de-France}

\begin{abstract}

We introduce the notion of an isotropic quantum state associated
with a Bohr-Sommerfeld manifold in the context of Berezin-Toeplitz
quantization of general prequantized symplectic manifolds, and we
study its semi-classical properties using the off-diagonal expansion
of the Bergman kernel. We then show how these results extend to the
case of non-compact orbifolds, and give an
application to relative Poincaré series in the theory of automorphic
forms.

\end{abstract}

\section{Introduction}

Let $(X,\om)$ be a compact symplectic manifold of dimension $2n$, and let $(L,h^L)$ be a Hermitian line bundle over $X$, endowed with a Hermitian connection $\nabla^L$ such that its curvature $R^L$ satisfies the following \emph{prequantization condition},
\begin{equation}\label{preq}
\om=\frac{\sqrt{-1}}{2\pi}R^L.
\end{equation}
Let $J$ be an almost complex structure on $TX$ compatible with $\om$, and let $g^{TX}$ be the Riemannian metric on $TX$ induced by $\om$ and $J$. For any $p\in\N^*$, we denote by $L^p$ the $p$-th tensor power of $L$, we write $\Delta^{L^p}$ for the associated Bochner Laplacian
acting on $\cinf(X,L^p)$, and consider the \emph{renormalized Bochner
Laplacian}, given for any $p\in\N^*$ by the formula
\begin{equation}\label{deltintro}
\Delta^{L^p}-2\pi np\,.
\end{equation}
Following \cite[(1.7)]{GU88}, it admits a discrete spectrum in $\R$,
and there exist constants
$\til{C},\,C,\,\mu>0$ such that for all $p\in\N^*$,
it has a finite number of eigenvalues contained in the interval
$[-\til{C},\til{C}]$, while all the others are bigger than
$\mu\,p-C$. Then for all $p\in\N^*$, we define the finite dimensional
space $\HH_p\subset\cinf(X,L^p)$ of \emph{almost holomorphic sections}
of $L^p$ as the direct sum of the eigenspaces associated with the
eigenvalues of the renormalized Bochner Laplacian
inside $[-\til{C},\til{C}]$. As explained in
\cref{subsetting}, these spaces satisfy the Riemann-Roch-Hirzebruch
formula for $p\in\N^*$ big enough, and are a natural
generalization of the spaces of holomorphic sections of $L^p$.

In fact, consider the special case of $J$ being integrable, making $(X,J,\om)$ into a \emph{Kähler manifold}, together with a holomorphic Hermitian line bundle $(L,h^L)$ such that its \emph{Chern connection} $\nabla^L$ (its unique Hermitian connection compatible with the holomorphic structure) satisfies \cref{preq}. For any $p\in\N^*$, writing $\dbar_p$ for the holomorphic $\dbar$-operator on forms with values in $L^p$ and $\dbar_p^*$ for its formal adjoint with respect to the $L^2$-Hermitian product, the \emph{Bochner-Kodaira} formula tells us that the operator \cref{deltintro} is equal to $2\dbar_p^*\dbar_p$. Then by a result of \cite[Th.~1.1]{BV89}, this operator shows a \emph{spectral gap}, so that the eigenvalues inside
$[-\til{C},\til{C}]$ are all equal to $0$ for $p\in\N^*$ big enough. The space $\HH_p$ of almost holomorphic sections considered above then reduces to the space of \emph{holomorphic sections} of $L^p$.
As explained for instance in \cite[\S\,9.2]{Woo92}, these
spaces can be thought as the spaces of quantum states of the
\emph{holomorphic quantization} of the symplectic manifold $(X,\om)$,
seen as a dynamical phase space of classical mechanics.
In this context, the integer $p\in\N^*$ represents a \emph{quantum
number}, usually inversely proportional to the \emph{Planck constant},
and asymptotic results when $p$ tends to infinity
describe the so-called \emph{semi-classical limit},
when the scale gets so large that we recover the laws of classical
mechanics as an approximation of the laws of quantum mechanics.
%
%

On the other hand, in the framework of geometric quantization associated with a regular Lagrangian fibration on $X$, the quantum states of $X$ are represented by immersed Lagrangian submanifolds $\iota:\Lambda\hookrightarrow X$ satisfying a property called the \emph{Bohr-Sommerfeld condition}, which asks for the existence of a non-vanishing section $\zeta\in\cinf(\Lambda,\iota^*L)$ parallel with respect to $\nabla^{\iota^*L}$ and satisfying
$|\zeta(x)|_{\iota^*L}=1$ for all $x\in\Lambda$
(see for example \cite{Sni75}). We call the data of $(\Lambda,\iota,\zeta)$ a \emph{Bohr-Sommerfeld Lagrangian}. The existence of a regular Lagrangian fibration on $X$ being too restrictive, we consider in general singular Lagrangian fibrations, in which we allow the dimension of the fibres to drop on a finite union of submanifolds of positive codimension in $X$. Removing the condition $\dim\Lambda=n$, we call the data of $(\Lambda,\iota,\zeta)$ a \emph{Bohr-Sommerfeld submanifold}. The typical case of a singular Lagrangian fibration is the case of \emph{toric manifolds}, where $X$ is endowed with an effective Hamiltonian action of $\IT^n=(S^1)^n$ and the fibres are given by the orbits of this action. For a comparison between holomorphic and real quantization in this context, see for example \cite{BFMN11}.

In this paper, we use the theory of the generalized Bergman kernel of Ma and Marinescu in \cite{MM08a} to study semi-classical properties of Bohr-Sommerfeld submanifolds in the context of the \emph{almost holomorphic quantization} described above. Here, the quantization of a Bohr-Sommerfeld submanifold is represented by a sequence $\{s_p\in\HH_p\}_{p\in\N^*}$, called an \emph{isotropic state}, defined for any $p\in\N^*$ by the formula
\begin{equation}\label{LSsimple}
s_{p}=\int_\Lambda P_p(x,\iota(y))\zeta^p(y)dv_\Lambda(y),
\end{equation}
where $dv_\Lambda$ is the Riemannian volume form of $(\Lambda,\iota^*g^{TX})$, $\zeta^p\in\cinf(\Lambda,\iota^*L^p)$ is the $p$-th tensor power of $\zeta$ and $P_p(\cdot,\cdot)$ is the \emph{generalized Bergman kernel}, that is the Schwartz kernel with respect to $dv_X$ of the orthogonal projection $P_p$ from $\cinf(X,L^p)$ to $\HH_p$ with respect to the natural $L^2$-Hermitian product. The expected behaviour of a quantum state in the semi-classical limit is to rapidly localize around the corresponding classical object, and we show in \cref{uopinf} that isotropic states indeed concentrate around the associated Bohr-Sommerfeld submanifold when $p\fl+\infty$. Furthermore, we establish in \cref{theonorme} the following estimate on the $L^2$-norm $\|\cdot\|_p$ of these sections
as $p\fl+\infty$, which is the first main result of this paper and which we state here in its simplest form.

\begin{theorem}\label{introtheonorme}
Let $(\Lambda,\zeta,\iota)$ be a Bohr-Sommerfeld submanifold of $X$. 
Then there exist $a_r\in\R,~r\in\N$, such that for any $k\in\N$ and as $p\fl+\infty$,
\begin{equation}
\norm{s_{p}}^2_p=p^{n-\frac{\dim\Lambda}{2}}\sum_{r=0}^k p^{-r}
a_r+O(p^{n-\frac{\dim\Lambda}{2}-(k+1)}).
\end{equation}
Furthermore, we have $a_0=2^{\frac{\dim\Lambda}{2}}\Vol(\Lambda)$, where
$\Vol(\Lambda)>0$ is the Riemannian volume of $(\Lambda,\iota^*g^{TX})$.
\end{theorem}

The proof of this Theorem uses the off-diagonal expansion expansion
as $p\fl +\infty$ of the generalized Bergman kernel given in
\cite[Th.\,1.19]{MM08a}, which
actually implies an analogous expansion
for the isotropic state $s_{p}$ around $\Lambda$ depending on the
position of the tangent spaces of $\Lambda$ with respect to
the Riemannian metric $g^{TX}$, similar to the asymptotic expansion
of the $G$-invariant Bergman kernel of Ma and Zhang in
\cite[Th.~0.2]{MZ08}. Although we do not state it explicitly,
this fact is implicitly used in
\cref{Lagint}, where we study the $L^2$-Hermitian product
$\<\cdot,\cdot\>_p$ of two such sections as $p\fl+\infty$. We show that this product tends rapidly to $0$ whenever the two associated submanifolds do not intersect, and we establish \cref{theointergal}, which is the second main result of this paper and which we state here in its simplest form, using the notion of \emph{clean intersection}
of \cref{cleandef}.

\begin{theorem}\label{introtheointergal}
Let $(\Lambda_1,\iota_1,\zeta_1)$ and $(\Lambda_2,\iota_2,\zeta_2)$ be two Bohr-Sommerfeld submanifolds with clean and connected intersection, and let $\{s_{j,p}\}_{p\in\N^*},j=1,2$, denote the associated isotropic states. Set $l=\dim \Lambda_1\cap\Lambda_2$ and $d_j=\dim\Lambda_j,~j=1,2$. Then there exist $b_{r}\in\C,~r\in\N$, such that for any $k\in\N$ and as $p\fl+\infty$,
\begin{equation}\label{introexp}
\<s_{1,p},s_{2,p}\>_p=p^{n-\frac{d_1+d_2}{2}+\frac{l}{2}}\lambda^p \sum_{r=0}^{k} p^{-r} b_{r}+O(p^{n-\frac{d_1+d_2}{2}+\frac{l}{2}-(k+1)}),
\end{equation}
where $\lambda\in\C$ is the value of the constant function on $\Lambda_1\cap\Lambda_2$ defined for any $x\in \Lambda_1\cap\Lambda_2$ by $\lambda(x)=\<\zeta_1(x),\zeta_2(x)\>_L$. Furthermore, if $\dim\Lambda_1=n$, the following formula holds,
\begin{equation}\label{introb0}
b_{0}=2^{\frac{n}{2}}\int_{\Lambda_1\cap\Lambda_2}
\det{}^{-\frac{1}{2}} \Big\{\sqrt{-1}\sum_{k=1}^{n-l}
h^{TX}(e_k,\nu_i)\om(e_k,\nu_j)\Big\}_{i,j=1}^{d_2-l}
|dv|_{\Lambda_1\cap\Lambda_2},
\end{equation}
where $\<e_i\>_{i=1}^{n-l},\<\nu_j\>_{j=1}^{d_2-l}$ are local orthonormal frames of the normal bundle of $\Lambda_1\cap\Lambda_2$ in $\Lambda_1,\Lambda_2$ respectively, and $|dv|_{\Lambda_1\cap\Lambda_2}$ is the Riemannian density on $\Lambda_1\cap\Lambda_2$ induced by $g^{TX}$.
\end{theorem}

The proof of this Theorem also gives a formula
for the first coefficient \cref{introb0} in the case $\Lambda_1$
and $\Lambda_2$ are both not Lagrangian, but
its geometric meaning is unclear, which is why we did not give it
explicitly. Note on the other hand
that although the integrand of \cref{introb0} is
nowhere vanishing, nothing prevents the whole integral to vanish in
general. In any case, this shows
that in the semi-classical limit, the Hermitian product of two isotropic states is closely related to the geometry of the intersection of the corresponding submanifolds. The left hand side of \cref{introexp} is called the \emph{intersection product} of $s_{1,p}$ and $s_{2,p}$, and can be thought as the cup product of some Lagrangian intersection theory (see \cite{Tyu00} for a discussion on this idea).

To give the most general formulation of \cref{introtheonorme,introtheointergal}, we use the theory of Berezin-Toeplitz operators for the generalized Bergman kernel on symplectic manifolds of \cite{ILMM17}, we consider any $J$-invariant Riemannian metric $g^{TX}$ on $TX$ and isotropic states taking values in an auxiliary Hermitian vector bundle $(E,h^E)$ with Hermitian connection $\nabla^E$. In the case of non-connected intersection, the expansion \cref{introexp} takes the form of a sum over the connected components. This is described in \cref{theonorme,theointergal},
in the case $X$ smooth and compact. 
When $(E,h^E)$ is the so-called \emph{metaplectic correction}, we 
recover the setting of metaplectic quantization and Lagrangian submanifolds endowed with half-forms, as explained in \cref{meta}.
On the other hand, the case of higher dimensional $(E,h^E)$
is relevant for the applications to relative Poincaré series
in the theory of vector-valued automorphic forms, as
explained below. In the same context, note that the metric $g^{TX}$
used in \cref{appli} is not the metric induced by $\om$ and $J$,
although the difference is rather trivial as noted in
the proof of \cref{thBScurve}. However, the case of a general
Hermitian metric $g^{TX}$ might be useful in the
study of relative Poincaré series on more general symmetric spaces.

In \cref{secnoncpctorbi}, we explain how the results of \cref{seclagstate} extend to the case of $(X,g^{TX})$ complete non-compact orbifold, when the immersed isotropic submanifold $\Lambda$ is compact and $(X,J,\om,g^{TX})$ is Kähler. As an application to the case where $X$ is the quotient of the Poincaré upper-half plane $\IH$ by a discrete subgroup $\Gamma$ of $\text{SL}_2(\R)$, we derive in \cref{appli} asymptotic results on relative Poincaré series in the theory of automorphic forms.

In the case $(X,J,\om,g^{TX})$ compact Kähler manifold with $c_1(TX)$ even, $E=\C$ and $\dim\Lambda_1=\dim\Lambda_2=n$, \cref{theointergal} is the main result of Borthwick, Paul and Uribe in \cite[Th.~3.2]{BPU95}, with an expansion in half-integer powers of $p$ in
\cite[(85)]{BPU95} instead of integer powers as in \cref{introexp}. This is explained in \cref{meta}, where we translate their use of the formalism of half-forms by taking for $E$ a square root of the canonical bundle of $X$. In the case where $\Gamma$ acts freely on $\IH$ and where $X=\IH/\Gamma$ is compact, the application to relative Poincaré series in \cref{appli} is the result of \cite[\S~4]{BPU95}. In the case where $(X,J,\om,g^{TX})$ is additionally equipped with an Hamiltonian action of a compact Lie group lifting to $(L,h^L,\nabla^L)$, an equivariant version of the results of \cite{BPU95} has been obtained by Debernardi and Paoletti \cite{DP06}. Semi-classical asymptotics on Lagrangian states have also been obtained by Charles in \cite{Cha03} in the case of discrete intersections and in the same particular context than in \cite{BPU95}.

The theory of Berezin-Toeplitz operators was first developed by Bordemann, Meinreken and Schlichenmaier in \citep{BMS94} and Schlichenmaier in \citep{Sch00} for the Kähler case, $E=\C$ and $g^{TX}(\cdot,\cdot)=\om(\cdot,J\cdot)$. The approach of both \cite{BMS94}, \cite{BPU95}, \cite{Cha03} and \cite{DP06} is based on the work of Boutet de Monvel and Sjöstrand on the Szegö kernel in \citep{BdMS75}, and the theory of Toeplitz structures developed by Boutet de Monvel and Guillemin in \citep{BdMG81}. Note that the definitions of \cref{BSsec} extend in a straightforward way to the case of \emph{spin$^c$ quantization} considered for example in \cite{MM08b}, and the results of \cref{seclagstate} and \cref{Lagint} certainly hold in this case. If $(X,J,\om,g^{TX})$ is further endowed with an Hamiltonian action of a compact Lie group $G$ lifting to $(L,h^L,\nabla^L),(E,h^E,\nabla^E)$ such that $0\in\Lie(G)^*$ is a regular point of the associated moment map $\mu:X\fl\Lie(G)^*$, and if $\iota:\Lambda\fl X$ intersects $\mu^{-1}(0)$ cleanly in the sense of \cref{cleandef}, then one can use the full off-diagonal expansion of the $G$-invariant Bergman kernel of Ma and Zhang in \cite[Th.~0.2, Rem.~0.3]{MZ08} to prove a result analogous to \cref{theonorme} for the $G$-invariant part of the associated isotropic state.

About relative Poincaré series,
Barron (previously Foth) studied in \cite{F02} the case of
Bohr-Sommerfeld
tori in higher dimensional symmetric spaces. The results of
\cref{secnoncpctorbi} can then be used
to generalize \cite[\S~1.3]{F02} to the case of non-compact or
orbifold symmetric spaces. In another direction, the results of 
\cref{secasylagstate,secclean} can be applied to study
relative Poincaré series associated
with isotropic submanifolds in higher dimensional symmetric spaces.
The case of geodesics on some specific compact 
quotients of the ball has been studied by Barron
in \cite{B18}. On the other hand, Alluhaibi and Barron
studied in \cite{AB17} the case of relative Poincaré series
associated with some submanifolds of the ball
which are not necessarily isotropic. Note that
they consider more generally
the case of vector-valued automorphic
forms, which corresponds for us to the case of $(E,h^E)$
flat Hermitian vector bundle of arbitrary dimension.
Our results can thus also
be applied to this case, when the underlying submanifold is isotropic.

A final motivation for this work is towards the program initiated by Witten in \cite{Wit89} in holomorphic quantization of Chern-Simons theory, showing an asymptotic expansion for Lagrangian states associated with some special Bohr-Sommerfeld Lagrangians inside the moduli space of flat connections on a Riemann surface, defined in \cite[Prop.~7.2]{JW92} and \cite[Prop.~3.27]{Free95}. Bohr-Sommerfeld Lagrangians in this context have also been studied by Tyurin in \cite{Tyu00}, and in the more general context of the Abelian Lagrangian Algebraic Geometry program of Gorodentsev and Tyurin \cite{GT01}. In both cases, it is of particular importance to be able to consider orbifolds.

\begin{ackn*}
This paper is a part of the author's PhD thesis under the supervision
of Pr. Xiaonan Ma, and the author wants to thank him for his
constant 
support. The author also wants to thank the anonymous
referees for useful comments
and suggestions. Part of this work was done while the author
was visiting the Institute
for Mathematical Sciences of Singapore, NUS, in May 2017, and
he would like to thank them for their hospitality.
This work was supported by the grant DIM-RDF
from Région Ile-de-France.
\end{ackn*}

\section{Generalized Bergman kernels on Symplectic Manifolds}\label{setting}

In this section, we set the context and notations, and recall the results of \cite{MM02}, \cite{MM08a} and \cite{ILMM17} we will need throughout the paper. We refer to the book \cite[Chap.4-8]{MM07} as a basic reference for the theory.

\subsection{Setting}\label{subsetting}

Let $(X,\om)$ be a compact symplectic manifold of dimension $2n$ with tangent bundle $TX$, and let $(L,h^L)$ be a Hermitian line bundle
over $X$, together with a Hermitian connection $\nabla^L$ satisfying \cref{preq}. Let $J$ be an almost complex structure compatible with
$\om$, and take $g^{TX}$ to be any $J$-invariant Riemannian metric
on $TX$. We write $\nabla^{TX}$ for the associated
Levi-Civita connection and $d^X(\cdot,\cdot)$ for the Riemannian distance of $(X,g^{TX})$

For any Euclidean vector bundle $(\EE,g^\EE)$, we write $\EE_\C$ for its complexification and still write $g^\EE$ for the induced $\C$-bilinear product on $\EE_\C$. Let us write
\begin{equation}\label{splitc}
TX_\C=T^{(1,0)}X\oplus T^{(0,1)}X
\end{equation}
for the splitting of $TX_\C$ into the eigenspaces of $J$ corresponding to the eigenvalues $\sqrt{-1}$ and $-\sqrt{-1}$ respectively. Then for any $x\in X,~v,w\in T^{(1,0)}_xX$, we define the positive Hermitian endomorphism $\dot{R}^L_x\in\End(T^{(1,0)}_xX)$ by the formula
\begin{equation}\label{RL}
g^{TX}(\dot{R}^L_x v,\overline{w})=R^L(v,\overline{w}).
\end{equation}
We denote by $K_X=\det(T^{*(1,0)}X)$ the canonical line bundle of $(X,J)$, endowed with the Hermitian structure and connection $h^{K_X},~\nabla^{K_X}$ induced by $g^{TX},~\nabla^{TX}$ via \cref{splitc}. We will consider as well the Riemannian metric $g^{TX}_\om$ on $TX$ defined by the formula
\begin{equation}\label{Jgras}
g^{TX}_\om(\cdot,\cdot)=\om(\cdot,J\cdot),
\end{equation}
and the Hermitian metric $h^{TX}_\om$ on $(TX,J)$ defined by
\begin{equation}\label{HermTX}
h^{TX}_\om=g^{TX}_\om-\sqrt{-1}\om.
\end{equation}
Note that if $g^{TX}=g^{TX}_\om$, then $\dot{R}^L=2\pi\Id_{T^{(1,0)}X}$. For any submanifold $Y\subset X$, we will write $g^{TY},g^{TY}_\om$ for the Riemannian metrics on $Y$ induced by $g^{TX},g^{TX}_\om$ and $dv_Y,dv_{Y,\om}$ for the induced Riemannian volume forms. In particular, we have
\begin{equation}\label{dv=dvom}
dv_{X,\om}=\det\left(\dot{R}^L/2\pi\right)dv_X.
\end{equation}
For any Hermitian vector bundle $(E,h^E)$ over $X$, we write 
$\<\cdot,\cdot\>_E$ and $|\cdot|_E$ for the Hermitian product and norm
induced by $h^E$.

Let $(E,h^E)$ be an auxiliary Hermitian vector bundle over $X$ with Hermitian connection $\nabla^E$, and write $R^E$ for the curvature
of $\nabla^E$. For any $p\in\N^*$, we write 
\begin{equation}\label{Ep}
E_p=L^p\otimes E,
\end{equation}
endowed with the Hermitan metric $h^{E_p}$  and connection$\nabla^{E_p}$ induced by $h^{L},\,h^{E}$ and $\nabla^{L},\,\nabla^{E}$.

\begin{defi}\label{bochner}
The \emph{Bochner Laplacian} $\Delta^{E_p}$ is the second order differential operator acting on $\cinf(X,E_p)$ by the formula
\begin{equation}\label{delta}
\Delta^{E_p}=-\sum_{j=1}^{2n}\left[(\nabla^{E_p}_{e_j})^2-\nabla^{E_p}_{\nabla^{TX}_{e_j}e_j}\right],
\end{equation}
where $\{e_j\}_{j=1}^{2n}$ is any local orthonormal frame of $TX$ with respect to $g^{TX}$.

For any $p\in\N^*$ and any Hermitian smooth section $\Phi\in\cinf(X,\End(E))$, the \emph{renormalized Bochner Laplacian} $\Delta_{p,\Phi}$ is the second order differential operator acting on $\cinf(X,E_p)$ by the formula
\begin{equation}\label{deltalpe}
\Delta_{p,\Phi}=\Delta^{E_p}- p\Tr[\dot{R}^L]+\Phi.
\end{equation}
\end{defi}
From now on, we fix $\Phi\in\cinf(X,\End(E))$ and simply write $\Delta_p$ for the associated renormalized Bochner Laplacian. In the Kähler case, if $g^{TX}=g^{TX}_\om$ and if $\Phi$ is equal to $-\sqrt{-1}R^E$ contracted with $\om$, we recover twice the Kodaira Laplacian of $E_p$. On the other hand, if $g^{TX}=g^{TX}_\om$ and $E=\C$, we recover \cref{deltintro}.

The \emph{$L^2$-Hermitian product} $\<\cdot,\cdot\>_p$ on $\cinf(X,E_p)$ is given for any $s_1,s_2\in\cinf(X,E_p)$ by the formula
\begin{equation}\label{L2}
\<s_1,s_2\>_p=\int_X \<s_1(x),s_2(x)\>_{E_p} dv_X(x).
\end{equation}
Let $\|\cdot\|_p$ be the associated $L^2$-norm, and
let $L^2(X,E_p)$ be the completion of $\cinf(X,E_p)$ with respect to 
$\|\cdot\|_p$. Then $\Delta_p$ is a self-adjoint second order differential operator on $L^2(X,E_p)$, and has discrete spectrum contained in $\R$. Furthermore, we have the following refinement of \cite[Th.~2a)]{GU88}.

\begin{theorem}{\cite[Cor.~1.2]{MM02}}\label{specdeltapphi}
There exist $\til{C},~C>0$ such that for all $p\in\N^*$,
\begin{equation}
\Spec(\Delta_p)\subset\ [-\til{C},\til{C}]\ \cup\ ]2\mu_0 p-C,+\infty[,
\end{equation}
where $\mu_0=\inf\limits_{x\in X,v\in T^{(1,0)}_xX} R^L_x(v,\overline{v})/g^{TX}_{x}(v,\overline{v})$.
\end{theorem}

For any $p\in\N^*$, we define the \emph{space of almost holomorphic sections} $\HH_p\subset L^2(X,E_p)$ of $E_p$ as the direct sum of the eigenspaces of $\Delta_p$ associated with the eigenvalues in $[-\til{C},\til{C}]$. Then by standard elliptic theory, we have $\HH_p\subset\cinf(X,E_p)$ and $\dim\HH_p<+\infty$. By \cite[Cor.1.2]{MM02}, for any $p\in\N^*$ big enough, the dimension of $\HH_p$ is computed by the Riemann-Roch-Hirzebruch formula,
and is in particular a polynomial of degree $n$ in $p$.
Note that by \cite[Cor.\,3.3]{MM08a}, the eigenvalues
in $[-\til{C},\til{C}]$ are not all equal to $0$ in general, and
for $p\in\N^*$ big enough, this happens
if and only if $(X,\om,J)$ is in fact Kähler.

Let $\pi_j:X\times X\fl X,~j=1,2$, denote the first and second projections. For any $p\in\N^*$,
we define a vector bundle over $X\times X$ by the formula
\begin{equation}
E_p\boxtimes E_p^*=\pi_1^*E_p\otimes\pi_2^*E_p^*.
\end{equation}
The orthogonal projection $P_p:\cinf(X,E_p)\fl\HH_p$ with respect to \cref{L2} has smooth Schwartz kernel $P_p(\cdot,\cdot)\in\cinf(X\times X,E_p\boxtimes E_p^*)$ with respect to $dv_X$, defined for any $s\in\cinf(X,E_p)$ and $x\in X$ by
\begin{equation}\label{ker}
(P_p s)(x)=\int_X P_p(x,y)s(y) dv_X(y).
\end{equation}
For any $F\in\cinf(X,\End(E))$, we define the \emph{Berezin-Toeplitz quantization} of $F$ as the family $\{T_{F,p}\}_{p\in\N^*}$ of operators acting on $\cinf(X,E_p)$ for any $p\in\N^*$ by
\begin{equation}\label{Toep}
T_{F,p}=P_pFP_p,
\end{equation}
where $F$ denotes the operator of pointwise application of
the endomorphism $F$. Then $T_{F,p}$ has smooth Schwartz kernel
$T_{F,p}(\cdot,\cdot)\in\cinf(X\times X,E_p\boxtimes E_p^*)$
with respect to $dv_X$, given for any $x,y\in X$ by
\begin{equation}\label{Toepker}
T_{F,p}(x,y)=\int_X P_p(x,w)F(w)P_p(w,y) dv_X(w).
\end{equation}

For any $\sigma>0$, we use the notation $O(p^{-\sigma})$ as $p\fl+\infty$ in the usual sense with respect to $|\cdot|_E$, uniformly in $x\in X$. The notation $O(p^{-\infty})$ means $O(p^{-\sigma})$ for any $\sigma>0$. Unless otherwise stated, we also use the convention to sum on free indices appearing twice in a single term.

\subsection{Local model}\label{locmod}

Let $(u,v):=(u_1,\dots,u_n,v_1,\dots,v_n)\in\R^{2n}$ be the canonical symplectic coordinates associated with the standard symplectic form $\Om$ on $\R^{2n}$ given by
\begin{equation}\label{Om}
\Om=\sum_{j=1}^n du_j\wedge dv_j.
\end{equation}
We write $\R^n\times\{0\}=\{(u,0)\in\R^{2n}~|~u\in\R^n\}$ and $\{0\}\times\R^n=\{(0,v)\in\R^{2n}~|~v\in\R^n\}$ for the two canonical oriented Lagrangian subspaces of $(\R^{2n},\Om)$ and write $\<\cdot,\cdot\>, |\cdot|$ for the canonical scalar product and norm of $\R^{2n}$. 
To match with the notations of \cite{MM08a}, we will write $Z:=(u,v)\in\R^{2n}$, and use the same notation for the radial vector field of $\R^{2n}$. For any $\epsilon>0$, we denote by $B^{\R^{2n}}(0,\epsilon)$ the ball of center $0$ and radius $\epsilon$ in $\R^{2n}$, and for any linear subspace $\Sigma\subset\R^{2n}$, we write $B^{\Sigma}(0,\epsilon):=B^{\R^{2n}}(0,\epsilon)\cap\Sigma$.

For any $m\in\N$, we write $|\cdot|_{\CC^m}$ for the local
$\CC^m$-norm on local sections of $E_p\boxtimes E_p^*$ over $X\times X$ induced by $h^L,~h^E,~\nabla^L,~\nabla^E$.

\begin{prop}{\cite[\S~1.1]{MM08a}}\label{theta}
For any $m,k\in\N,~\epsilon>0$ and $\theta\in\ ]0,1[$, there is $C_{m,k,\theta,\epsilon}>0$ such that for all $p\in\N^*$, and $x,x'\in X$ satisfying $d^X(x,x')>\epsilon p^{-\theta/2}$,
\begin{equation}
|P_p(x,x')|_{\CC^m}\leq C_{m,k,\theta,\epsilon}p^{-k}.
\end{equation}
%
\end{prop}
Let us now take $x_0\in X,~\epsilon_0>0,~V\subset X$ open neighbourhood of $x_0$ and
\begin{equation}\label{defphi}
\phi_{x_0}:B^{\R^{2n}}(0,\epsilon_0)\subset\R^{2n}\fl V
\end{equation}
a diffeomorphism sending $0$ to $x_0$, such that its differential at $0$ identifies $\Om$ and $\<\cdot,\cdot\>$ on $\R^{2n}$ with $\om$ and $g^{TX}_\om$ on $T_{x_0}X$. Let us make such a choice of diffeomorphisms \cref{defphi} for any $x_0$ in a small open set, smoothly in $x_0$. We cover $X$ with such open sets, and choose $\epsilon_0>0$ which does not depend on $x_0\in X$. As two Riemannian metrics induce equivalent distances in a continuous way with respect to parameters, there exist $0<a<b$ such that for any $x_0\in X$ and $Z,Z'\in B^{\R^{2n}}(0,\epsilon_0)$,
\begin{equation}\label{gcnt}
a|Z-Z'|<d^X(\phi_{x_0}(Z),\phi_{x_0}(Z'))<b|Z-Z'|.
\end{equation}
%
%
Then by \cref{gcnt}, we get the following corollary of \cref{theta}.

\begin{cor}\label{thetagal}
For any $\epsilon>0,~m,k\in\N$ and $\theta\in\ ]0,1[$, there is $C_{m,k,\theta,\epsilon}>0$ such that for all $x_0\in X,~p\in\N^*$ and $Z,Z'\in B^{\R^{2n}}(0,\epsilon_0)$ such that $|Z-Z'|>\epsilon p^{-\theta/2}$,
\begin{equation}
|P_p(\phi_{x_0}(Z),\phi_{x_0}(Z'))|_{\CC^m}\leq C_{m,k,\theta,\epsilon'}p^{-k}.
\end{equation}
\end{cor}

We use the following explicit local model on $\R^{2n}$ for the Bergman kernel, as found in \cite[(3.25)]{MM08b} for any $Z,Z'\in\R^{2n}$,
\begin{equation}\label{PPreal}
\begin{split}
\PP_{x_0}(Z,Z')=\exp\left(-\frac{\pi}{2} |Z-Z'|^2-\pi\sqrt{-1}\Om(Z,Z')\right).
\end{split}
\end{equation}
Note that the difference of \cref{PPreal} with \cite[(3.25)]{MM08b} comes from the fact that we are working with symplectic coordinates
$Z\in\R^{2n}$ adapted to $\om$ via \cref{defphi} instead of metric
coordinates adapted to $g^{TX}$
via the exponential map as in \cite[\S\ 3.2]{MM08b}.

Let $dZ$ be the canonical Lebesgue measure of $\R^{2n}$, and define
the smooth function $\kappa_{x_0}\in\cinf(B^{\R^{2n}}(0,\epsilon_0),\R)$ such that for any $Z\in B^{\R^{2n}}(0,\epsilon_0)$
in the chart \cref{defphi},
\begin{equation}\label{defkappa}
dv_{X}(Z)=\kappa_{x_0}(Z)dZ,~~\text{with}~~
\kappa_{x_0}(0)=\det\left(\dot{R}^L_{x_0}/2\pi\right)^{-1}.
\end{equation}
In the chart \cref{defphi}, we identify $E$, $L$ over $B^{\R^{2n}}(0,\epsilon_0)$ with $E_{x_0}, L_{x_0}$ through parallel transport with respect to $\nabla^E, \nabla^L$ along radial lines of $B^{\R^{2n}}(0,\epsilon_0)$. For any $x_0$ in a small open set, we identify $L_{x_0}$ with $\C$ using any unit local frame of $L$. 

For any $f\in\cinf(X,E)$, we write $f_{x_0}\in\cinf(B^{\R^{2n}}(0,\epsilon_0),E_{x_0})$ for the restriction of $f$ to $B^{\R^{2n}}(0,\epsilon_0)$ in this trivialization. Similarly, for any $T_p(\cdot,\cdot)\in\cinf(X\times X, E_p\boxtimes E_p^*)$, we denote by $T_{p,x_0}(Z,Z')\in\End(E_{x_0})$ its image evaluated at $Z,Z'\in B^{\R^{2n}}(0,\epsilon_0)$ in this trivialization. If $Q(Z,Z')$ is a polynomial in $Z,Z'\in\R^{2n}$, we write $Q\PP_{x_0}(Z,Z'):=Q(Z,Z')\PP_{x_0}(Z,Z')$. 

Recall that we chose a family of charts $\{\phi_{x_0}\}_{x_0\in W}$ as in \cref{defphi} smoothly in $x_0\in W$, where $W$ is a small open set of $X$. Then $P_{p,x_0}(Z,Z')$ can be seen as a smooth section of $\pi^*\End(E)$ over $W\times B^{\R^{2n}}(0,\epsilon_0)\times B^{\R^{2n}}(0,\epsilon_0)$ evaluated in $x_0\in W,~Z,Z'\in B^{\R^{2n}}(0,\epsilon_0)$, where $\pi:W\times B^{\R^{2n}}(0,\epsilon_0)\times B^{\R^{2n}}(0,\epsilon_0)\fl W$ is the first projection. Let us write $|\cdot|_{\CC^m(X)}$ for the local $\CC^m$-norm on local sections of $\pi^*\End(E)$ induced by $h^E$ and derivation by $\nabla^{\pi^*\End(E)}$ in the direction of $x_0\in W$. We are now ready to state the following result, which was first proved in \cite[Th.~4.18']{DLM06} in the case of the spin$^c$ Dirac operator, and which in the following form comes essentially from \cite[Th.~2.1]{LMM16}.

\begin{lem}\label{asy}
For any $m,k\in\N$, $\epsilon>0$ and $\delta\in\ ]0,1[$, there is $C>0$ and $\theta\in\ ]0,1[$ such that for all $x_0\in X,~p\in\N^*$ and $|Z|,|Z'|<\epsilon p^{-\theta/2}$, 
\begin{multline}
\Big|p^{-n} P_{p,x_0}(\phi_{x_0}(Z),\phi_{x_0}(Z'))\\
-\sum_{r=0}^k p^{-r/2}J_{r,x_0}\PP_{x_0}(\sqrt{p}Z,\sqrt{p}Z')\kappa_{x_0}^{-1/2}(Z)\kappa_{x_0}^{-1/2}(Z')\Big|_{\CC^m(X)}
\leq Cp^{-\frac{k+1}{2}+\delta},
\end{multline}
where $\{J_{r,x_0}(Z,Z')\}_{r\in\N}$ is a family of polynomials in $Z,Z'\in\R^{2n}$ of the same parity as $r$ and with values in $\End(E_{x_0})$, depending smoothly on $x_0\in X$. Furthermore, we have
\begin{equation}\label{J0}
J_{0,x_0}(Z,Z')\equiv\Id_{E_{x_0}}.
\end{equation}

\end{lem}

Parallel to \cref{theta} and \cref{asy}, we have the following result on the asymptotic expansion as $p\fl+\infty$ of the Berezin-Toeplitz operator \cref{Toep}. It was first proved in \cite[Lemma~4.6]{MM08b} in the spin$^c$ case, and in this form comes essentially from \cite[Lemma~3.3]{ILMM17}.

\begin{lem}\label{Toepasy}
Let $F\in\cinf(X,\End(E))$. Then for any $0<\epsilon\leq \epsilon_0,~m,k\in\N$ and $\theta\in\ ]0,1[$, there is $C_{m,k,\theta,\epsilon}>0$ such that for all $x_0\in X,~p\in\N^*,~Z,Z'\in\R^{2n}$, $|Z-Z'|>\epsilon p^{-\theta/2}$,
\begin{equation}
|T_{F,p}(\phi_{x_0}(Z),\phi_{x_0}(Z'))|_{\CC^m}\leq C_{m,k,\theta,\epsilon}p^{-k}.
\end{equation}
Furthermore, for any $m,k\in\N$, $\epsilon>0$ and $\delta\in]0,1[$, there is $C>0$ and $\theta\in]0,1[$ such that for all $x_0\in X,~p\in\N^*$, $|Z|,|Z'|<\epsilon p^{-\theta/2}$, 
\begin{multline}\label{Toepasyexp}
\Big|p^{-n} T_{F,p,x_0}(\phi_{x_0}(Z),\phi_{x_0}(Z'))\\
-\sum_{r=0}^k p^{-r/2} \Q_{r,x_0}\PP_{x_0}(\sqrt{p}Z,\sqrt{p}Z')\kappa_{x_0}^{-1/2}(Z)\kappa_{x_0}^{-1/2}(Z')\Big|_{\CC^m(X)}
\leq Cp^{-\frac{k+1}{2}+\delta},
\end{multline}
where $\{\Q_{r,x_0}(Z,Z')\}_{r\in\N}$ is a family of polynomials in $Z,Z'\in\R^{2n}$ of the same parity as $r$ and with values in $\End(E_{x_0})$, depending smoothly on $x_0\in X$. Furthermore, we have
\begin{equation}\label{Q0}
\Q_{0,x_0}(Z,Z')\equiv F_{x_0}.
\end{equation}
\end{lem}

\subsection{Gaussian integrals}
\label{Gausssec}

We now recall some well-known facts about Gaussian integrals, which will be used for local computations in the next sections. For any $k\in\N^*$, let $\<\cdot,\cdot\>$ denote the canonical scalar product of $\R^k$. For any positive symmetric matrix $C$ acting on $\R^k$, we recall the following classical formula for the Gaussian integral,
\begin{equation}\label{gaussintgal}
\int_{\R^k} \exp(-\pi \<Z,CZ\>) dZ=\det{}^{-\frac{1}{2}} C.
\end{equation}

By analytic continuation, this formula is still valid when $C$ is a symmetric matrix with complex coefficients, providing the integral is well defined along a path in the space of symmetric matrices joining $C$ with a real positive symmetric matrix. Specifically, for $A$ positive symmetric matrix and $B$ real symmetric matrix, we will consider the path
\begin{equation}\label{continuation}
\begin{split}
\gamma:[0,1] & \rightarrow\GL_k(\C)\\
t & \mapsto A +t\sqrt{-1}B.
\end{split}
\end{equation}

Then \cref{gaussintgal} holds for $C=A+\sqrt{-1}B$, with the determination of the square root given by continuation along the image of \cref{continuation} by the application $\det^{-1}:GL_n(\C)\fl\C$. Henceforth, we will always use this determination of the square root of the determinant for $C=A+\sqrt{-1}B$ as above.


\section{Isotropic states}\label{seclagstate}
Through all this section, we use the context and notations of \cref{setting}. In particular, recall that $(X,\om)$ is a compact symplectic manifold of dimension $2n$, and that the curvature of $\nabla^L$ on $(L,h^L)$ over $X$ satisfies \cref{preq}.

\subsection{Bohr-Sommerfeld submanifolds}\label{BSsec}

An immersed submanifold $\iota:\Lambda\fl X$ is said to be \emph{isotropic} if $\iota^*\om=0$. If in addition $\dim\Lambda=n$, it is said to be \emph{Lagrangian}. Let $\nabla^{\iota^*L},\,h^{\iota^*L}$ be the connection and Hermitian metric induced by $\nabla^L,\,h^L$ on the pullback line bundle $\iota^*L$ over $\Lambda$. Note that by \cref{preq}, the condition $\iota^*\om=0$ implies that $\nabla^{\iota^*L}$ is \emph{flat}. This observation motivates the following definition.

\begin{defi}\label{BS}
A properly immersed oriented isotropic submanifold $\iota:\Lambda\fl X$ is said to satisfy the \emph{Bohr-Sommerfeld condition} if there exists a non-vanishing smooth section $\zeta\in\cinf(\Lambda,\iota^* L)$ satisfying
\begin{equation}\label{nabs=0}
\nabla^{\iota^*L} \zeta=0.
\end{equation}
Taking $\zeta$ satisfying further $|\zeta(x)|_{\iota^* L}=1$ for any $x\in\Lambda$, the data of $(\Lambda,\iota,\zeta)$ is called a \emph{Bohr-Sommerfeld submanifold} of $X$, or a \emph{Bohr-Sommerfeld Lagrangian} if in addition $\dim\Lambda=n$.
\end{defi}

Note that the properness hypothesis on $\iota$ implies that
$\Lambda$ is compact. Furthermore, this definition depends only on the symplectic structure on $(X,\om)$ and the prequantization condition \cref{preq} on $(L,h^L,\nabla^L)$. As $\nabla^L$ is Hermitian, up to renormalisation we can always assume that $\zeta\in\cinf(\Lambda,\iota^* L)$ satisfying \cref{nabs=0} is such that $|\zeta(x)|_{\iota^*L}=1$ for any $x\in\Lambda$. Finally, from the compactness of $X$, the properness hypothesis on $\iota$ is equivalent to the compactness of $\Lambda$.

\begin{rem}\label{hol}
As noted above, if $\iota:\Lambda\fl X$ is isotropic, then $\nabla^{\iota^*L}$ is flat over $\Lambda$, hence determined by its holonomy $\hol_{\iota^*L}:\pi_1(\Lambda)\fl S^1\subset\C$. We can then reformulate \cref{nabs=0} by saying that $\iota:\Lambda\fl X$ satisfies the Bohr-Sommerfeld condition if and only if $\hol_{\iota^*L}=\{1\}$. Now if the order of $\hol_{\iota^*L}$ is finite, then there exists a finite covering $j:\hat{\Lambda}\fl\Lambda$ such that $\hol_{j^*\iota^*L}=\{1\}$, so that $\iota\circ j:\hat{\Lambda}\fl X$ satisfies the Bohr-Sommerfeld condition. In particular, if there is $k\in\N$ such that $\iota:\Lambda\fl X$ satisfies the Bohr-Sommerfeld condition for $L^k$ instead of $L$, then the order of $\hol_{\iota^*L}$ divides $k$, thus is finite. Such a $\iota:\Lambda\fl X$ is called a \emph{Bohr-Sommerfeld submanifold of order $k$}, and up to finite covering, \cref{BS} also accounts for these. In the same line of thought, if $\iota:\Lambda\fl X$ is not orientable, we can always work on the orientation double cover of $\Lambda$.
\end{rem}

Let us now set some notations. We write $\iota^L,~\iota^E$ and $\iota_p$ for the natural maps covering $\iota:\Lambda\fl X$ on the respective total spaces of $L,~E$ and $E_p$ for any $p\in\N^*$. If $\zeta$ is any section of $\iota^*L$, we write $\zeta^p$ for the $p$-th power of $\zeta$ defined as a section of $\iota^*L^p$. If additionally $f$ is a section of $\iota^*E$, we write $\zeta^pf$ for the induced tensor product in $\iota^*E_p$.

From now on, we fix an almost complex structure $J$ on $TX$ compatible with $\om$, an auxiliary Hermitian vector bundle $(E,h^E)$ with Hermitian connection $\nabla^E$ and a $J$-invariant Riemannian
metric $g^{TX}$ on $TX$. We write $g^{T\Lambda}:=\iota^* g^{TX}$ for
the induced Riemannian metric on $T\Lambda$, and $dv_\Lambda$ for
the Riemannian volume form of $(\Lambda,g^{T\Lambda})$.
Recall that $\Lambda$ is compact by hypothesis.

\begin{defi}\label{Lagstate} The \emph{isotropic state} associated with $(\Lambda,\iota,\zeta)$ and $f\in\cinf(\Lambda,\iota^*E)$ is the family of sections $\{s_{f,p}\in\HH_p\}_{p\in\N^*}$ defined for any $x\in X$ by the formula
\begin{equation}\label{defLagstate}
s_{f,p}(x)=\int_{\Lambda} P_p(x,\iota(y))\iota_p.\zeta^pf(y)dv_{\Lambda}(y).
\end{equation}
\end{defi}

As $\iota$ is locally an embedding, when working locally we will often omit the mention of $\iota$, considering locally $\Lambda$ as a submanifold of $X$. With this convention, equation \cref{defLagstate} becomes
\begin{equation}\label{defLagstate1}
s_{f,p}(x)=\int_{\Lambda} P_p(x,y)\zeta^pf(y)dv_{\Lambda}(y).
\end{equation}
We list the basic properties of isotropic states in the following proposition, which holds for any $p\in\N^*$.

\begin{prop}\label{proprepgal}
For any $f_1, f_2\in\cinf(\Lambda,\iota^*E)$, we have the following additivity property,
\begin{equation}\label{add}
s_{f_1+f_2,p}=s_{f_1,p}+s_{f_2,p}.
\end{equation}
For any $s\in\HH_p$, we have the following reproducing property,
\begin{equation}\label{rep}
\<s,s_{f,p}\>_p=\int_{\Lambda} \<s(\iota(x)),\iota_p.\zeta^pf(x)\>_{E_p}\,dv_{\Lambda}(x).
\end{equation}
For any $f\in\cinf(\Lambda,\iota^*E)$ and any $F\in\cinf(X,\End(E))$, the action of $T_{F,p}$ on $s_{f,p}$ is given for any $x\in X$ by the formula
\begin{equation}\label{Toeplag}
T_{F,p}s_{f,p}=\int_{\Lambda} T_{F,p}(x,\iota(y))\iota_p.\zeta^pf(y)dv_{\Lambda}(y).
\end{equation}
\end{prop}
\begin{proof}
First, the additivity property \cref{add} is obvious from \cref{defLagstate}. Next, recall that $P_p$ is self-adjoint with respect to $\<\cdot,\cdot\>_p$ for any $p\in\N^*$, and restricts to the identity of $\HH_p$. Then using \cref{ker}, \cref{defLagstate1} and Fubini, we compute for any $s\in\HH_p$,
\begin{equation}\label{comprep}
\begin{split}
\<s,s_{f,p}\>_p & =\int_X\left<s(y),\int_{\Lambda} P_p(y,\iota(x))\iota_p.\zeta^pf(x)dv_{\Lambda}(x)\right>_{E_p} dv_X(y)\\
& =\int_{\Lambda} \left<\int_X P_p(\iota(x),y)s(y)dv_X(y),\iota_p.\zeta^pf(x)\right>_{E_p} dv_{\Lambda}(x)\\
& =\int_{\Lambda} \<s(\iota(x)),\iota_p.\zeta^pf(x)\>_{E_p}\ dv_{\Lambda}(x).
\end{split}
\end{equation}

The reproducing property \cref{rep} follows from \cref{comprep}. Finally, from \cref{Toep}, we get for any $f\in\cinf(\Lambda,\iota^*E)$ and $F\in\cinf(X,\End(E))$ that $T_{F,p}s_{f,p}=P_pFs_{f,p}$. Then by \cref{Toepker}, \cref{defLagstate} and using Fubini, we get for any  $x\in X$,
\begin{equation}\label{compToeplag}
\begin{split}
(T_{F,p}s_{f,p})(x) & =\int_X\int_\Lambda P_p(x,w)F(w)P_p(w,\iota(y))\iota_p.\zeta^pf(y)dv_\Lambda(y)dv_X(w)\\
& =\int_\Lambda T_{F,p}(w,\iota(y))\iota_p.\zeta^pf(y)dv_\Lambda(y).
\end{split}
\end{equation}

From \cref{compToeplag}, we get \cref{Toeplag}.
\end{proof}

\subsection{Asymptotic expansion of isotropic states}\label{secasylagstate}

In this section, we establish the first semi-classical properties of isotropic states. In particular, we show that the $L^2$-norm of an isotropic state admits an asymptotic expansion as $p\fl+\infty$,
and we compute the highest order term.

For any $p\in\N^*$, we write $|\cdot|_{E_p}$ for the norm on $E_p$ induced by $h^L$ and $h^E$. We show in the following proposition how an isotropic state concentrates around the image of the associated isotropic submanifold as $p\fl+\infty$. 

\begin{prop}\label{uopinf}
Let $f\in\cinf(\Lambda,\iota^*E)$. For any closed subset $K\subset X$ such that $K\cap\iota(\Lambda)=\0$ and for any $k\in\N$, there exists $C_k>0$ such that for any $x\in K$ and all $p\in\N^*$,
\begin{equation}
|s_{f,p}(x)|_{E_p}<C_k p^{-k}.
\end{equation}
\end{prop}
\begin{proof}
This is a direct consequence of \cref{theta} and formula \cref{defLagstate}.
\end{proof}
Recall that for any $p\in\N^*$, we write $\|\cdot\|_p$ for the norm on
$\cinf(X,E_p)$ induced by $\<\cdot,\cdot\>_p$, and we write
$|\cdot|_{\iota^*E}$ for the norm on $\iota^*E$ over $\Lambda$
induced by $h^E$. The rest of the section is dedicated to the proof
of the following theorem.

\begin{theorem}\label{theonorme}
Let $f\in\cinf(\Lambda,\iota^*E)$. Then there exist $a_r\in\R,~r\in\N$, such that for any $k\in\N$ and as $p\fl+\infty$,
\begin{equation}\label{norme}
\norm{s_{f,p}}^2_p=p^{n-\frac{\dim\Lambda}{2}}\sum_{r=0}^k p^{-r} a_r+O(p^{n-\frac{\dim\Lambda}{2}-(k+1)}),
\end{equation}
with first coefficient $a_0\in\R$ given by the formula
\begin{equation}\label{b0norme}
a_0=2^{\frac{\dim\Lambda}{2}}\int_{\Lambda}|f|_{\iota^*E}^2
\det\left(\dot{R}^L_{x_0}/2\pi\right)\frac{dv_{\Lambda}}
{dv_{\Lambda,\om}}dv_{\Lambda}.
\end{equation}
Additionally, for any $F\in\cinf(X,\End(E))$, the product $\<T_{F,p}s_{f,p},s_{f,p}\>_p$ satisfies the expansion of \cref{norme} with $a_r\in\C,~r\in\N$, and
\begin{equation}\label{Toepb0norme}
a_0=2^{\frac{\dim\Lambda}{2}}\int_{\Lambda}\<Ff,f\>_{\iota^*E}\det\left(\dot{R}^L_{x_0}/2\pi\right)\frac{dv_{\Lambda}}{dv_{\Lambda,\om}}dv_{\Lambda}.
\end{equation}
\end{theorem}
\begin{proof}
Note first that the reproducing property \cref{rep} gives
\begin{equation}\label{repnorme}
\norm{s_{f,p}}_p^2=\int_\Lambda\<s_{f,p}(\iota(x)),\zeta^pf(x)\>_{E_p}dv_\Lambda(x).
\end{equation}
Using \cref{repnorme}, we are reduced to evaluate $s_{f,p}$ on the image of $\iota:\Lambda\fl X$. Let then $x_0\in X$ be in the image of $\iota$. As $\iota:\Lambda\fl X$ is an immersion, there is an integer $m\in\N$ such that for any small enough connected neighbourhood $V$ of $x_0$ in $X$, there are $m$ disjoint connected open sets $U_1,\dots, U_m\subset\Lambda$ such that $\iota^{-1}(V)=\cup_{j=1}^mU_j$. Using \cref{theta}, we can then localize the problem in the following way
as $p\fl+\infty$,
\begin{equation}\label{loclagm}
\begin{split}
s_{f,p}(x_0) & =\int_{\Lambda} P_p(x_0,\iota(x))\zeta^pf(x)dv_{\Lambda}(x)\\
& =\sum_{j=1}^m \int_{U_j} P_p(x_0,\iota(x))\zeta^pf(x)dv_{\Lambda}(x)+O(p^{-\infty}).
\end{split}
\end{equation}
In view of \cref{norme,repnorme,loclagm}, we can assume that $f$ has compact support around $\cup_{j=1}^mU_j$. Using \cref{add,repnorme}, we are reduced further to the case where $f$ has compact support around one of the $U_j$ for some $j$. As $U:=U_j$ is embedded in $X$ through $\iota$, we can consider $U$ as a submanifold of $X$, and \cref{loclagm} translates to
\begin{equation}\label{loclag}
s_{f,p}(x_0)=\int_U P_p(x_0,x)\zeta^pf(x)dv_\Lambda(x)+O(p^{-\infty}).
\end{equation}
By definition of $U$ as a submanifold of $X$, we can
take $\phi_{x_0}:B^{\R^{2n}}(0,\epsilon)\fl V$ with
$\epsilon>0,~V\subset X$ as in \cref{defphi} such that $\phi_{x_0}$ identifies $U\subset V$ with $B^{\Sigma}(0,\epsilon)$, where $\Sigma$ is a vector subspace of $\R^{2n}$. Then $\Sigma$ is an isotropic subspace of $(\R^{2n},\Om)$. We identify $E$, $L$ over $B^{\R^{2n}}(0,\epsilon)$ with $E_{x_0}$, $L_{x_0}$ as in \cref{locmod}. In particular, we use the unitary vector $\zeta(x_0)$ to identify $L_{x_0}$ with $\C$ , where $\zeta\in\cinf(\Lambda,\iota^*L)$ is the section associated with $(\Lambda,\iota,\zeta)$ as in \cref{BS}. As $\zeta$ is parallel with respect to $\nabla^{\iota^*L}$ along $\Lambda$, it is identified with $1\in\C$ over $B^{\Sigma}(0,\epsilon)$ in this trivialization.
Let $du$ be the Lebesgue measure of $\Sigma$, and define the function $h\in\cinf(B^{\Sigma}(0,\epsilon),\R)$ for all $u\in B^{\Sigma}(0,\epsilon)$ by
\begin{equation}\label{volcoord}
dv_\Lambda(u)=h(u) du,~~\text{with}~~
h(0)=(dv_\Lambda/dv_{\Lambda,\om})(x_0).
\end{equation}
Using \cref{thetagal}, \cref{asy} and \cref{defkappa},
for any $\delta\in\ ]0,1[$, we get $\theta\in\ ]0,1[$ such that
as $p\fl+\infty$,
\begin{equation}\label{estlag}
\begin{split}
\<s_{f,p}&(x_0),\zeta^p f(x_0)\>_{E_p}\\
= &\int_{B^{\Sigma}(0,\epsilon p^{-\theta/2})} \<P_p(x_0,\phi_{x_0}(u))\zeta^pf(\phi_{x_0}(u)),\zeta^pf(x_0)\>_{E_p}dv_\Lambda(u)+O(p^{-\infty})\\
= &p^n \int_{B^{\Sigma}(0,\epsilon p^{-\theta/2})} \sum_{r=0}^k p^{-r/2}\<J_{r,x_0}\PP_{x_0}(0,\sqrt{p}u)f_{x_0}(u),f(x_0)\>_E\\
&\kappa_{x_0}^{-1/2}(u)\kappa_{x_0}^{-1/2}(0)dv_\Lambda(u)
+p^n\int_{B^{\Sigma}(0,\epsilon p^{-\theta/2})} O(p^{-\frac{k+1}{2}+\delta})dv_\Lambda(u)+O(p^{-\infty})\\
= &p^n \int_{B^{\Sigma}(0,\epsilon p^{-\theta/2})}
\det\left(\dot{R}^L_{x_0}/2\pi\right)^{\frac{1}{2}}
\sum_{r=0}^k p^{-r/2}\<J_{r,x_0}\PP_{x_0}(0,\sqrt{p}u)f_{x_0}(u),f(x_0)\>_E\\
&\kappa_{x_0}^{-1/2}(u)h(u)du
+p^np^{-\frac{\theta\dim\Lambda}{2}}O(p^{-\frac{k+1}{2}+\delta}).
\end{split}
\end{equation}

Let us write $g_{x_0}=h\kappa_{x_0}^{1/2}f_{x_0}\in\cinf(B^{\Sigma}(0,\epsilon),E_{x_0})$. Then from \cref{defkappa} and \cref{volcoord}, we get the following Taylor expansion in $u\in\R^n$ up to order $k\in\N$,
\begin{multline}\label{Taylorfk}
g_{x_0}(u)  =(h\kappa_{x_0}^{-1/2}f_{x_0})(0)+\sum_{1\leq|\alpha|\leq k}\frac{\partial^{|\alpha|} g_{x_0}}{\partial u^\alpha}\frac{u^\alpha}{\alpha!}+O(|u|^{k+1})\\
=f(x_0)(dv_\Lambda/dv_{\Lambda,\om})(x_0)
\det\left(\dot{R}^L_{x_0}/2\pi\right)^{\frac{1}{2}}
+\sum_{1\leq |\alpha|\leq k}p^{-\alpha/2}\frac{\partial^{|\alpha|} g_{x_0}}{\partial u^\alpha}\frac{(\sqrt{p}u)^\alpha}{\alpha!}\\
+p^{-\frac{k+1}{2}}O(|\sqrt{p}u|^{k+1}).
\end{multline}
On the other hand, recall from \cref{asy} that $J_{r,x_0}(0,\sqrt{p}u)\in\End(E_{x_0})$ is polynomial in $\sqrt{p}u$ of the same parity as $r\in\N$. Let $M_k$ be the supremum of the degree of $J_{r,x_0}$ for all $1\leq r\leq k$, and write $\delta'=\delta+(M_k+k+1+d)(1-\theta)/2$. We deduce from \cref{estlag} and \cref{Taylorfk} the existence of a sequence $\{G_r\}_{r\in\N}$ of polynomials in one variable of $\R^n$ of the same parity as $r$, with values in $\C$, and with
\begin{equation}\label{G0ZZ'}
G_0\equiv
|f(x_0)|_E^2\frac{dv_\Lambda}{dv_{\Lambda,\om}}(x_0)
\det\left(\dot{R}^L_{x_0}/2\pi\right),
\end{equation}
such that as $p\fl+\infty$,
\begin{equation}\label{finestu0}
\begin{split}
&\<s_{f,p}(x_0), \zeta^pf(x_0)\>_{E_p}\\
& =p^n \sum_{r=0}^k p^{-r/2}\int_{B^{\Sigma}(0,\epsilon p^{-\theta/2})} G_r(\sqrt{p}u)\PP_{x_0}(0,\sqrt{p}u)du+O(p^{n-
\frac{\dim\Lambda+k+1}{2}+\delta'})\\
& = p^{n-\frac{\dim\Lambda}{2}} \sum_{r=0}^k p^{-r/2} \int_{B^{\Sigma}(0,\epsilon p^{(1-\theta)/2})} G_r(u)\PP_{x_0}(0,u)du
+O(p^{n-\frac{\dim\Lambda+k+1}{2}+\delta'}).
\end{split}
\end{equation}
Recall from \cref{PPreal} that
\begin{equation}\label{Pox}
\PP_{x_0}(0,u)=\exp\left(-\frac{\pi}{2}|u|^2\right),
\end{equation}
so that as $1-\theta>0$, the integral of $\PP_{x_0}(0,u)$
over $\R^n\backslash B^{\Sigma}(0,\epsilon p^{(1-\theta)/2})$
with respect to $u$
decreases exponentially as $p\fl+\infty$, and we then deduce from \cref{finestu0} that
\begin{equation}\label{finestu}
\<s_{f,p}(x_0),\zeta^pf(x_0)\>_{E_p} = p^{n-\frac{d}{2}} \sum_{r=0}^k p^{-r/2} \int_{\Sigma} G_r(u)\PP_{x_0}(0,u)du+
O(p^{n-\frac{\dim\Lambda+k+1}{2}+\delta'}).
\end{equation}
As $G_r$ is of the same parity as $r$, we immediately deduce from \cref{Pox} that for any $m\in\N$,
\begin{equation}\label{odd}
\int_{\Sigma}G_{2m+1}(u)\PP_{x_0}(0,u)du=0.
\end{equation}
Finally,
we get from \cref{G0ZZ'} and \cref{Pox}
the following formula for the highest
order term of \cref{finestu},
\begin{multline}\label{b0}
\int_{\Sigma}G_0(u)\PP_{x_0}(0,u)du\\
 =|f(x_0)|_E^2(dv_\Lambda/dv_{\Lambda,\om})(x_0)\det\left(\dot{R}^L_{x_0}/2\pi\right)\int_{\Sigma}\exp(-\frac{\pi}{2}|u|^2)du\\
 =2^{\frac{\dim\Lambda}{2}}|f(x_0)|_E^2(dv_\Lambda/dv_{\Lambda,\om})(x_0)\det\left(\dot{R}^L_{x_0}/2\pi\right).
\end{multline}

Then recalling that all the estimates above are uniform in
$x_0\in\iota(\Lambda)$, and by \cref{repnorme,odd,dv=dvom},
it suffices to integrate \cref{finestu} and \cref{b0} over
$x_0\in\iota(\Lambda)$ with respect to $dv_{\Lambda}$ to get \cref{norme}
and \cref{b0norme}.

Using the property \cref{Toeplag},
the proof of the asymptotic expansion as $p\fl+\infty$ of
$\<T_{F,p}s_{f,p},s_{f,p}\>_p$ is completely analogous to the proof
of the asymptotic expansion of $\|s_p\|_p$,
simply replacing the polynomials $J_{r,x_0}$ of \cref{asy}
by the polynomials $\QQ{r}$ of \cref{Toepasy} in the computations
above. This achieves the proof of \cref{theonorme}.
\end{proof}

\section{Isotropic intersections}\label{Lagint}
Let us consider two Bohr-Sommerfeld submanifolds $(\Lambda_j,\iota_j,\zeta_j)$ together with $f_j\in\cinf(\Lambda_j,\iota^*_j E)$, for $j=1,2$, and set $d_j=\dim\Lambda_j$. In this section, we establish the existence of an asymptotic expansion as $p\fl+\infty$ of the Hermitian product $\<s_{f_1,p},s_{f_2,p}\>_p$ of the two associated isotropic states, and we compute the highest order term, which depends only on the geometry of the intersection. Note that the case $\{s_{f_1,p}\}_{p\in\N^*}=\{s_{f_2,p}\}_{p\in\N^*}$ is precisely the result of \cref{theonorme}.

We will need the following regularity assumption, that we will
use throughout the section.

\begin{defi}\label{cleandef}
We say that two proper immersions $\iota_j:\Lambda_j\fl X,~j=1,2$, are intersecting \emph{cleanly} if for any $x\in\iota_1(\Lambda_1)\cap\iota_2(\Lambda_2)$ and any $y_j\in \Lambda_j$ such that $\iota_1(y_1)=\iota_2(y_2)=x$, there exist neighbourhoods
$U_j\subset\Lambda_j$ of $y_j$ such that the intersection $\iota_1(U_1)\cap\iota_2(U_2)$ is a submanifold of $X$ satisfying $T_x \iota_1(U_1)\cap T_x\iota_2(U_2)=T_x(\iota_1(U_1)\cap\iota_2(U_2))$.
\end{defi}

We then define the \emph{intersection} of
two immersions $\iota_1:\Lambda_1\fl X$ and $\iota_2:\Lambda_2\fl X$ over $X$ as their fibred product $\Lambda_1\cap\Lambda_2$, which
comes with two immersion
$j_i:\Lambda_1\cap\Lambda_2\fl\Lambda_i,~i=1,2$, such that
$\iota_1\circ j_1=\iota_2\circ j_2$ and which are
universal for this property.
Under the assumption of \cref{cleandef} above, this fibred product
has a natural smooth structure. In fact, consider smooth
atlases $\UU_1,\,\UU_2$
of $\Lambda_1,\,\Lambda_2$ respectively, such that for any
$U_j\in\UU_j,\,j=1,2$, the immersion $\iota_j$
restricted to $U_j$ is an embedding satisfying the assumption of
\cref{cleandef} as soon as the intersection is non-empty. We can
then define an atlas of $\Lambda_1\cap\Lambda_2$ as the set of all
intersections $U_1\cap U_2$ for all $U_1\in\UU_1$, $U_2\in\UU_2$,
with transition maps induced by the ones of $\UU_1$ and $\UU_2$.

Note that this definition of intersection is local, and
reduces to the usual one in the case of embeddings.
For that reason, one can readily reduce to the usual definition
of a clean intersection when working locally.
A typical situation when this general definition is
needed is in the natural
case when $\iota_1:\Lambda_1\fl X$ and $\iota_2:\Lambda_2\fl X$ are
Bohr-Sommerfeld submanifolds of respective
order $k_1\in\N^*$ and $k_2\in\N^*$ in the sense \cref{hol}, with
$k_1$ and $k_2$ prime with each other.

\subsection{Asymptotic expansion of discrete intersections}\label{sectrans}
In this section, we deal with the case of discrete intersections. We consider first the easy case when the intersection is empty.

\begin{prop}\label{propnonint}
Suppose that \text{$\Lambda_1\cap\Lambda_2=\0$}, and let $F\in\cinf(X,\End(E))$. Then for any $k\in\N$, there exists $C_k>0$ such that for all $p\in\N^*$,
\begin{equation}\label{propnonint1}
|\<T_{F,p}s_{f_1,p},s_{f_2,p}\>_p|<C_kp^{-k}.
\end{equation}
\end{prop}
\begin{proof}
Using the reproducing property \cref{rep}, we get for any $p\in\N^*$,
\begin{equation}\label{comppropnonint}
\<T_{F,p}s_{f_1,p},s_{f_2,p}\>_p=\int_\Lambda \<T_{F,p}s_{f_1,p}(\iota_2(x)),\zeta_2^pf_2(x)\>_{E_p} dv_{\Lambda_2}(x).
\end{equation}
In particular, as $\Lambda_2$ is compact by hypothesis, we can choose
$K=\iota_2(\Lambda_2)$ in \cref{uopinf}, and
we deduce \cref{propnonint1} from \cref{comppropnonint}.
\end{proof}

In view of \cref{propnonint}, we will assume from now on that $\Lambda_1\cap\Lambda_2$ is not empty. In the statement of the following theorem, the immersions $\iota_i:\Lambda_i\fl X$ and $j_i:\Lambda_1\cap\Lambda_2\fl\Lambda_i,~i=1,2,$ are implicit, and we omit to mention them for simplicity.
%

\begin{theorem}\label{theointer}
Suppose that $(\Lambda_1,\iota_1,\zeta_1)$ and $(\Lambda_2,\iota_2,\zeta_2)$ intersect cleanly, and that their intersection $\Lambda_1\cap\Lambda_2$ in the sense above is discrete. Set $m=\#\,\Lambda_1\cap\Lambda_2$ and write $\Lambda_1\cap\Lambda_2=\{x_1,\dots,x_m\}$. Then for any $F\in\cinf(X,\End(E))$, there exist $b_{q,r}\in\C,~r\in\N,~1\leq q\leq m$, such that for any $k\in\N$ and as $p\fl+\infty$,
\begin{equation}\label{u1u2}
\<T_{F,p}s_{f_1,p},s_{f_2,p}\>_p=p^{n-\frac{d_1+d_2}{2}}\sum_{q=1}^m\lambda^p_{q} \sum_{r=0}^k p^{-r} b_{q,r}+O(p^{n-\frac{d_1+d_2}{2}-(k+1)}),
\end{equation}
where $\lambda_{q}=\<\zeta_1(x_q),\zeta_2(x_q)\>_L$. Furthermore, if $\dim\Lambda_1=n$, we have
\begin{multline}\label{bmj0}
b_{q,0}=2^{n/2}\<F_{x_q}f_1(x_q),f_2(x_q)\>_{x_q}
\det\left(\dot{R}^L_{x_q}/2\pi\right)
\frac{dv_{\Lambda_1}}{dv_{\Lambda_1,\om}}
\frac{dv_{\Lambda_2}}{dv_{\Lambda_2,\om}}(x_q)\\
\det{}^{-\frac{1}{2}}\Big\{\sqrt{-1}\sum_{k=1}^nh^{TX}_\om(e_k,\nu_i)\om(e_k,\nu_j)\Big\}_{i,j=1}^{d_2},
\end{multline}
where $\<e_i\>_{i=1}^n,\<\nu_j\>_{j=1}^{d_2}$ are oriented orthonormal bases for $g^{TX}_\om$ of the tangent spaces of $\Lambda_1,\Lambda_2$ in $X$ at $x_q$, and the square root of the determinant is determined by \cref{continuation}.
\end{theorem}
\begin{proof}
We will prove \cref{theointer} for $F=\Id_E$ (so that $T_{F,p}=P_p$), the proof of the general case being totally analogous by \cref{Toepasy} and property \cref{Toeplag}. First, using the
reproducing property \cref{rep}, we get for any $p\in\N^*$,
\begin{equation}\label{compreptrans}
\<s_{f_1,p},s_{f_2,p}\>_p=\int_\Lambda \<s_{f_1,p}(\iota_2(x)),\zeta_2^pf_2(x)\>_{E_p} dv_{\Lambda_2}(x).
\end{equation}
We can then reproduce the argument in the proof of \cref{propnonint}, using \cref{uopinf} to reduce the proof to the case of $f_2$ with compact support in a given
neighbourhood of $\iota_2^{-1}(\iota_1(\Lambda_1)\cap\iota_2(\Lambda_2))$, which is a finite set by assumption. Symmetrically, using the reproducing property of $s_{f_1,p}$ instead of $s_{f_2,p}$, we can assume further that $f_1$ has compact support in a given neighbourhood of $\iota_1^{-1}(\iota_1(\Lambda_1)\cap\iota_2(\Lambda_2))$. By the additivity property \cref{add,u1u2},
we are further reduced to the case of $f_i$ with compact support in a neighborhood in a single point $y_i\in\Lambda_i$, for each $i=1,2$. Using \cref{uopinf}, we are finally reduced to the case $\iota_1(y_1)=\iota_2(y_2)$.
Set $x_0:=\iota_1(y_1)=\iota_2(y_2)\in X$.

Let $U_j\subset\Lambda_j$ be as in \cref{cleandef}, intersecting
cleanly at $x_0\in X$ only. In particular,
using the \cref{Lagstate} of an isotropic state,
equation \cref{compreptrans} becomes, as $p\fl+\infty$,
\begin{equation}\label{loclagm1}
\begin{split}
\<s_{1,p},s_{2,p}\>_p & =\int_{U_2} \<s_{f_1,p}(x),\zeta_2^pf_2(x)\>_{E_p} dv_{\Lambda_2}(x)+O(p^{-\infty})\\
& =\int_{U_2}\int_{U_1} \<P_p(x,y)\zeta_1^pf_1(y),\zeta_2^pf_2(x)\>_{E_p}dv_{\Lambda_1}(y)dv_{\Lambda_2}(x)+O(p^{-\infty}).
\end{split}
\end{equation}
By definition of $U_1$ as a submanifold of $X$, we can consider a
chart as in \cref{defphi} in which $U_1$ is identified with
$B^{\Sigma_1}(0,\epsilon)$ for some $\epsilon>0$, where $\Sigma_1$
is an isotropic space of $(\R^{2n},\Om)$. Now in this chart,
we can consider a projection $\pi_{1,2}:\R^{2n}\fl\Sigma_2$
preserving $\Sigma_1$, where $\Sigma_2$
it the tangent space to $U_2$ at $x_0$ in this chart, and use
it to identify $U_2$ with $B^{\Sigma_2}(0,\epsilon)$. In
that way, we can construct
$\phi_{x_0}:B^{\R^{2n}}(0,\epsilon)\fl V$ with
$\epsilon>0,~V\subset X$ as in \cref{defphi}, with
$V\subset X$ such that $V\cup \Lambda_j=U_j$,
such that $\varphi_{x_0}$ identifies $U_j$ with
$B^{\Sigma_j}(0,\epsilon)$ for any $j=1,2$,
where $\Sigma_1$ and $\Sigma_2$ are isotropic subspaces of
$(\R^{2n},\Om)$. As $U_1$ and $U_2$ intersect cleanly at $x_0$ only,
we have $\Sigma_1\cap\Sigma_2=\{0\}$. We identify $E$, $L$ over $B^{\R^{2n}}(0,\epsilon)$ with $E_{x_0}$, $L_{x_0}$ as in \cref{locmod}, and use the unitary vector $\zeta_1(x_0)$ to identify $L_{x_0}$ with $\C$. Then $\zeta_1$ is identified with $1\in\C$ over $B^{\Sigma_1}(0,\epsilon)$ in this trivialization. As $\zeta_2$ is parallel with respect to $\nabla^{\iota_2^*L}$ over $U_2$, it is identified with $\bar{\lambda}\in\C$ over $B^{\Sigma_2}(0,\epsilon)$, where $\lambda=\<\zeta_1(x_0),\zeta_2(x_0)\>_L$.

Then  as $p\fl+\infty$, equation \cref{loclagm1} becomes
\begin{multline}\label{loclag1}
\<s_{1,p},s_{2,p}\>_p=\lambda^p\int_{B^{\Sigma_2}(0,\epsilon)}\int_{B^{\Sigma_1}(0,\epsilon)}\<P_p(\phi_{x_0}(Z),\phi_{x_0}(Z'))f_{1,x_0}(Z'),f_{2,x_0}(Z)\>_E\\
dv_{\Lambda_1}(Z')dv_{\Lambda_2}(Z)+O(p^{-\infty}).
\end{multline}
Let $du$ and $dw$ be the Lebesgue measures of $\Sigma_1$ and $\Sigma_2$ respectively. For any $j=1,2$, define the functions $h_j\in\cinf(B^{\Sigma_j}(0,\epsilon),\R)$ in the chart \cref{defphi} for any $u\in B^{\Sigma_1}(0,\epsilon),~w\in B^{\Sigma_2}(0,\epsilon)$ by
\begin{equation}\label{volcoord12}
dv_{\Lambda_1}(u)=h_1(u)du\quad\text{and}\quad dv_{\Lambda_2}(w)=h_2(w)dw,
\end{equation}
with $h_j(0)=(dv_{\Lambda_j}/dv_{\Lambda_j,\om})(x_0)$ for $j=1,2$. Recalling \cref{gcnt} and the fact that $|\lambda^p|=1$ for all $p\in\N^*$, we can use \cref{thetagal} and \cref{asy}, to get $\theta\in\ ]0,1[$ for any $k\in\N,~\delta\in\ ]0,1[$, such that
as $p\fl+\infty$, equation \cref{loclag1} becomes
\begin{equation}\label{loclag2}
\begin{split}
\<s_{1,p}, & s_{2,p}\>_p =\lambda^p\int_{B^{\Sigma_2}(0,\epsilon p^{-\theta/2})}\int_{B^{\Sigma_1}(0,\epsilon p^{-\theta/2})}\\
& \<P_p(\phi_{x_0}(Z),\phi_{x_0}(Z'))f_{1,x_0}(Z'),f_{2,x_0}(Z)\>_E dv_{\Lambda_1}(Z')dv_{\Lambda_2}(Z)+O(p^{-\infty})\\
= & \lambda^p p^n\sum_{r=0}^k  p^{-r/2}\int_{B^{\Sigma_2}(0,\epsilon p^{-\theta/2})}\int_{B^{\Sigma_1}(0,\epsilon p^{-\theta/2})}\\
& \quad\quad\quad\quad\<J_{r,x_0}\PP_{x_0}(\sqrt{p}Z,\sqrt{p}Z')f_{1,x_0}(Z'),f_{2,x_0}(Z)\>_E\\
& \quad\quad\quad\quad\quad\quad\quad\quad\quad\quad\kappa_{x_0}^{-1/2}(Z')\kappa_{x_0}^{-1/2}(Z)dv_{\Lambda_1}(Z')dv_{\Lambda_2}(Z)\\
&+p^n\int_{B^{\Sigma_2}(0,\epsilon p^{-\theta/2})}\int_{B^{\Sigma_2}(0,\epsilon p^{-\theta/2})}O(p^{-\frac{k+1}{2}+\delta})dv_{\Lambda_1}(Z')dv_{\Lambda_2}(Z)+O(p^{-\infty})\\
= & \lambda^p p^n\sum_{r=0}^k p^{-r/2}\int_{B^{\Sigma_2}(0,\epsilon p^{-\theta/2})}\int_{B^{\Sigma_1}(0,\epsilon p^{-\theta/2})}\\
& \quad\quad\quad\<J_{r,x_0}\PP_{x_0}(\sqrt{p}w,\sqrt{p}u)f_1(u),f_2(w)\>_E\\
& \quad\quad\quad\quad\quad\kappa_{x_0}^{-1/2}(u)\kappa_{x_0}^{-1/2}(w)h_1(u)h_2(w)dudw+p^np^{-\frac{(d_1+d_2)\theta}{2}}O(p^{-\frac{k+1}{2}+\delta}).
\end{split}
\end{equation}
Consider now the Taylor expansion up to order $k\in\N$ of $g_j=h_j\kappa_{x_0}^{-1/2}f_{j,x_0}\in\cinf(B^{\Sigma_j}(0,\epsilon),\C)$ for $j=1,2$ as in \cref{Taylorfk}. By \cref{asy} and formula
\cref{defkappa}, following the proof of \cref{theonorme},
we get $\delta'>0$ and a sequence $\{G_r\}_{r\in\N}$ of polynomials in two variables of $\R^{2n}$ with values in $\C$, of the same parity as $r$ with
\begin{equation}\label{F0}
G_0\equiv\<f_1(x_0),f_2(x_0)\>_E\frac{dv_{\Lambda_1}}{dv_{\Lambda_1,\om}}\frac{dv_{\Lambda_2}}{dv_{\Lambda_2,\om}}(x_0)
\det\left(\dot{R}^L_{x_0}/2\pi\right),
\end{equation}
such that as $p\fl+\infty$, equation \cref{loclag2} becomes
\begin{multline}\label{loclag3}
\<s_{1,p},s_{2,p}\>_p=\lambda^p p^{n-\frac{d_1+d_2}{2}} \sum_{r=1}^k p^{-r/2}\int_{B^{\Sigma_2}(0,\epsilon p^{(1-\theta)/2})}\int_{B^{\Sigma_1}(0,\epsilon p^{(1-\theta)/2})}\\
G_r\PP_{x_0}(w,u)du dw+O(p^{n-\frac{d_1+d_2+k+1}{2}+\delta'}).
\end{multline}
As $\Sigma_1\cap\Sigma_2=\{0\}$, we get from \cref{PPreal} that
\begin{equation}\label{PP<}
\left|\PP_{x_0}(w,u)\right|\leq \exp\big(C(|u|+|w|)\big),
\end{equation}
for some $C>0$ and all $w\in\Sigma_1$ and $u\in\Sigma_2$.
In particular, as $1-\theta>0$, its integral
in $u\in\Sigma_1\backslash B^{\Sigma_1}(0,\epsilon p^{(1-\theta)/2})$ and in
$w\in\Sigma_2\backslash B^{\Sigma_2}(0,\epsilon p^{(1-\theta)/2})$ decrease exponentially and uniformly as $p\fl+\infty$.
Equation \cref{loclag3} then becomes
\begin{equation}\label{localag4}
\<s_{1,p},s_{2,p}\>_p=\lambda^p \sum_{r=1}^k p^{-r/2}\int_{\Sigma_2}\int_{\Sigma_1}G_r\PP_{x_0}(w,u)du dw+O(p^{-\frac{k+1}{2}+\delta'}).
\end{equation}
Let us now evaluate the integrals in \cref{localag4}. Up to linear symplectic transformation, the canonical symplectic basis $\{e_j,f_j\}_{j=1}^n$ of $(\R^{2n},\Om)$ can be chosen such that $\Sigma_1=\<e_1,\dots, e_{d_1}\>$ as an oriented isotropic subspace. Let $\nu_1,\dots, \nu_{d_2}\in\Sigma_2$ form an oriented orthonormal basis of $\Sigma_2$ for the metric induced by $\<\cdot,\cdot\>$. Consider the matrices $A$ and $B$ given by
\begin{equation}\label{ABdef}
\begin{split}
A=(a_i^j)_{1\leq i\leq n,1\leq j\leq d_2}\quad & \text{with}\quad a_i^j=\Om(e_i,\nu_j),\\
B=(b_i^j)_{1\leq i\leq n,1\leq j\leq d_2}\quad & \text{with}\quad b_i^j=\<e_i,\nu_j\>.
\end{split}
\end{equation}
As $\Om(e_i,\nu_j)=\<f_i,\nu_j\>$ for all $1\leq i\leq n,\,
1\leq j\leq d_2$, we know that for any $1\leq j\leq d_2$,
\begin{equation}\label{viejfj}
\nu_j=\sum_{i=1}^n b_i^je_i+\sum_{i=1}^n a_i^jf_i.
\end{equation}
Let us write $dt:=dt_1\dots dt_{d_2}$ for the Lebesgue measure of $\R^{d_2}$, and let $\varphi$ be any measurable function with compact support on $\R^{2n}$. Setting $w=t_i\nu_i$ for any $w\in\Sigma_2$, integration of $\varphi$ along $\Sigma_2$ for its Lebesgue measure $dw$ becomes
\begin{equation}\label{intf}
\int_{\Sigma_2} \varphi(w) dw=\int_{\R^{d_2}} \varphi\Big(\sum_{j=1}^{d_2}t_j\nu_j\Big) dt.
\end{equation}
%
Let us use the convention of \cref{subsetting}, summing $i$ from $1$ to $d_1$ and $k,j$ from $1$ to $d_2$ whenever they appear as free indices. From the explicit expression \cref{PPreal}, taking Fourier transform and performing a change of variables, we compute
\begin{equation}\label{comput1}
\begin{split}
& \int_{\Sigma_2}\int_{\Sigma_1}G_r(w,u)\PP_{x_0}(w,u)du dw=\int_{\R^{d_2}}\int_{\R^{d_1}}G_r(t_j\nu_j,u_ie_i)\PP_{x_0}(t_j\nu_j,u_ie_i)du dt\\
& =\int_{\R^{d_2}}\int_{\R^{d_1}} G_r(t_j\nu_j,u_ie_i)
\exp\left(-\frac{\pi}{2}\sum\limits_{i=d_1+1}^n
\left((t_jb_i^j)^2+(t_ja_i^j)^2\right)\right)\\
&\exp\left(-\frac{\pi}{2}\sum\limits_{i=1}^{d_1} \left((u_i-t_jb_i^j)^2+(t_ja_i^j)^2+2\sqrt{-1}u_i t_ja_i^j\right)\right)du dt\\
& =\int_{\R^{d_2}}\int_{\R^{d_1}} G_r(t_j\nu_j,(u_i+t_jb^j_i)e_i)
\exp\left(-\frac{\pi}{2}\sum\limits_{i=d_1+1}^n ((t_jb_i^j)^2+(t_ja_i^j)^2)\right)\\
&\exp\left(-\frac{\pi}{2}\sum\limits_{i=1}^{d_1}\left(u_i^2+(t_ja_i^j)^2+2\sqrt{-1}u_i t_ja_i^j+2\sqrt{-1} t_kb_i^ka_i^jt_j\right)\right)du dt\\
& =2^{d_2/2}\int_{\R^{d_2}} \til{G}_r(t)\exp\left(-\frac{\pi}{2}\sum\limits_{i=d_1+1}^n \left((t_jb_i^j)^2+(t_ja_i^j)^2\right)\right)\\
&\exp\left(-\pi\sum\limits_{i=1}^{d_1} \left((t_ja^j_i)^2+\sqrt{-1} t_kb_i^ka_i^jt_j\right)\right)dt,\\
\end{split}
\end{equation}
where $\til{G}_r(t)$ are polynomials in $t\in\R^{d_1}$ of the same parity as $r$. Using that $\Sigma_1\cap\Sigma_2=\{0\}$, we get the convergence of the integral in \cref{comput1}, and as the integrand is of the same parity as $r$, the integral vanishes if $r$ is odd. Together with \cref{localag4}, this proves \cref{u1u2}.

Let us now compute the first coefficient of \cref{localag4} in the case $\dim\Lambda_1=n$. From \cref{comput1}, we get
\begin{multline}\label{comput2}
\int_{\Sigma_2}\int_{\Sigma_1}\PP_{x_0}(u,w)du dw
=2^{n/2}\int_{\R^{d_2}} \exp\left(-\pi\sum\limits_{i=1}^n \left((t_ja^j_i)^2+\sqrt{-1} t_kb_i^ka_i^jt_j\right)\right)dt.
\end{multline}
As $\<\nu_1,\dots,\nu_{d_2}\>$ is the basis of an isotropic submanifold, we get that $\om(\nu_j,\nu_k)=0$ for all $1\leq j,k\leq d_2$, which is equivalent through \cref{viejfj} to the fact that $B^TA$ is symmetric. Then summing $i$ from $1$ to $n$, the matrix $(a_i^ka_i^j+\sqrt{-1}b^k_ia_i^j)_{k,j=1}^{d_2}=A^TA+\sqrt{-1}B^TA$ is symmetric, and its real part $A^TA$ is strictly positive as $A$ has maximal rank. Thus from \cref{comput2} and using \cref{gaussintgal}, we get
\begin{equation}\label{comput2bis}
\int_{\Sigma_2}\int_{\Sigma_1}\PP_{x_0}(u,w)du dw =2^{n/2}\det{}^{-\frac{1}{2}}(\sqrt{-1}(B-\sqrt{-1}A)^TA).
\end{equation}
Then the formula \cref{bmj0} for the first coefficient follows from \cref{ABdef,comput2,HermTX,F0,dv=dvom}.
\end{proof}

\subsection{Asymptotic expansion of clean intersections}\label{secclean}

In this section, we deal with the case of general clean intersection in the sense of \cref{cleandef}. The main difference with
\cref{theointer} is that the coefficients of the expansion are
now given as an integral over the fixed point set. The main
additional difficulty is then to show that one can in fact
split the integral between an integral over the fixed point set
and an integral over transversal slices, then
integrate the expansion of \cref{asy} over the transversal slices
following the proof of \cref{theointer}.

As in \cref{sectrans}, the immersions $\iota_i:\Lambda_i\fl X$ and $j_i:\Lambda_1\cap\Lambda_2\fl\Lambda_i,~i=1,2,$ are implicit in the statement of the following theorem, and we omit to mention them for simplicity.

\begin{theorem}\label{theointergal}
Suppose that $(\Lambda_1,\iota_1,\zeta_1)$ and $(\Lambda_2,\iota_2,\zeta_2)$ intersect cleanly. Let $\Lambda_1\cap\Lambda_2=\cup_{q=1}^mY_m$ be the decomposition into connected components of their intersection in the sense above, and set $l_q=\dim Y_q$. Then for any $F\in\cinf(X,\End(E))$, there exist $b_{q,r}\in\C,~r\in\N,~1\leq q \leq m$, such that for any $k\in\N$ and as $p\fl+\infty$,
\begin{equation}\label{<u1u2>}
\<T_{F,p}s_{f_1,p},s_{f_2,p}\>_p=\sum_{q=1}^m p^{n-\frac{d_1+d_2}{2}+\frac{l_q}{2}}\lambda^p_q \sum_{r=0}^{k} p^{-r} b_{q,r}+O(p^{n-\frac{d_1+d_2}{2}+\frac{l_q}{2}-(k+1)}),
\end{equation}
where $\lambda_q\in\C$ is the value of the constant function on $Y_q$ defined for any $x\in Y_q$ by $\lambda_q(x)=\<\zeta_1(x),\zeta_2(x)\>_L$. If $\dim\Lambda_1=n$, we have
\begin{multline}\label{bm0}
b_{q,0}=2^{n/2}\int_{Y_q} \left<F f_1(x),f_2(x)\right>_E\det{}^{1/2}\left(\dot{R}^L/2\pi\right)\frac{dv_{\Lambda_2}}{dv_{\Lambda_2,\om}}(x)\\
\det{}^{-\frac{1}{2}} \Big\{\sqrt{-1}\sum_{k=1}^{n-l_q}h^{TX}_\om(e_k,\nu_i)\om(e_k,\nu_j)\Big\}_{i,j=1}^{d_2-l_q}(x)|dv|_{Y_q,\om}(x),
\end{multline}
where $\<e_i\>_{i=1}^{n-l_q},\<\nu_j\>_{j=1}^{d_2-l_q}$ are local orthonormal frames of the normal bundle of $Y_q$ inside $\Lambda_1,~\Lambda_2$ with respect to $g^{T\Lambda_1}_\om,g^{T\Lambda_2}_\om$, and $|dv|_{Y_q,\om}$ is the Riemannian density of $(Y_q,g^{TY_q}_\om)$. The square root of the determinant is determined by \cref{continuation}.
\end{theorem}
\begin{proof}
Let us set $F=\Id_E$, the proof of the general case being totally analogous by \cref{Toepasy} and \cref{Toeplag}. Using \cref{uopinf}, \cref{add,<u1u2>}, we can assume that $\Lambda_1\cap\Lambda_2$ has a unique connected component $Y$, and that $f_j,~j=1,2,$ have compact support in a given open set. The following computations are then local on $Y$, and we may assume $Y$ oriented and embedded in $\Lambda_2$ by $j_2:Y\fl\Lambda_2$. We omit the mention of $j_2$ in the sequel. We set $l=\dim Y$.

Let $N$ be the \emph{normal bundle} of $Y$ inside $\Lambda_2$, identified with the orthogonal complement of $TY$ in $(T\Lambda_2,g^{T\Lambda_2}_\om)$, and let $g^N_\om$ be the induced metric on $N$.
%
%
Let $\epsilon>0$ be such that the exponential map $\exp^{\Lambda_2}_\om$ of $(\Lambda_2,g^{T\Lambda_2}_\om)$ restricted to $B^N(0,\epsilon):=\{w\in N~|~|w|_{g^{N}}<\epsilon\}$ is a diffeomorphism on its image. Then for any $x\in Y$ and with $Y$ embedded in $N$ as its zero section, the differential $d\exp^{\Lambda_2}_{\om,x}:T_xY\oplus N_x\rightarrow T_x\Lambda_2$ is the identity map, and $\exp^{\Lambda_2}_\om(B^N(0,\epsilon))$ is a tubular neighbourhood of $Y$ in $\Lambda_2$.

Let $dw$ be an Euclidean volume form on the fibres of $(N,g^N_\om)$ such that the volume form $dw dv_{Y,\om}$ on the total space of $N$ is compatible with the orientation of $X$. Let $h_2\in\cinf(B^N(0,\epsilon),\R)$ be defined via the exponential map for any $x\in Y,~w\in N_x$ with $|w|_{g^N_{\om,x}}<\epsilon$ by
\begin{equation}\label{volE1}
dv_{\Lambda_2}(x,w)=h_2(x,w)~dw dv_{Y,\om}(x).
\end{equation}
Then $h_2(x,0)=(dv_{\Lambda_2}/dv_{\Lambda_2,\om})(x)$.
%
Let us now define $I(f_1,f_2)\in\cinf(B^N(0,\epsilon),\C)$ at $x\in Y,w\in N_x$ with $|w|_{g^N_{\om,x}}<\epsilon$, by the formula
\begin{multline}\label{midint}
I(f_1,f_2)(x,w)=\int_{\Lambda_1}\<P_p((x,w),\iota_1(y))\iota_{1,p}.\zeta_1^pf_1(y),\zeta_2^pf_2(x_0,w)\>_{E_p}\\
 h_2(x_0,w)dv_{\Lambda_1}(y).
\end{multline}

Using \cref{defLagstate}, \cref{add}, \cref{rep,uopinf}, we get from \cref{volE1,midint},
\begin{multline}\label{intmidint}
\<s_{1,p},s_{2,p}\>_p = \int_{\Lambda_2}\int_{\Lambda_1} \<P_p(\iota_2(x),\iota_1(y))\iota_{1,p}.\zeta_1^pf_1(y),\iota_{2,p}.\zeta_2^pf_2(x)\>_{E_p}dv_{\Lambda_1}(y)dv_{\Lambda_2}(x)\\
= \int_{\exp^{\Lambda_2}_\om(B^N(0,\epsilon))} \int_{\Lambda_1} \<P_p(x,\iota_1(y))\iota_{1,p}.\zeta_1^pf_1(y),\zeta_2^pf_2(x)\>_{E_p}\\
dv_{\Lambda_1}(y)dv_{\Lambda_2}(x)+O(p^{-\infty})\\
=\int_{x\in Y}\int_{B^{N_x}(0,\epsilon)} I(f_1,f_2)(x,w) dwdv_{Y,\om}(x)+O(p^{-\infty}).
\end{multline}
Fix now $x_0\in Y$.
%
%
%
%
%
%
Take $\epsilon>0,~U\subset\Lambda_1$ and a diffeomorphism $\phi_{x_0}^{\Lambda_1}:B^{\R^{d_1}}(0,\epsilon)\fl U$ sending $0$ to $x_0$ and such that its differential at $0$ identifies $\<\cdot,\cdot\>$ with $g^{T\Lambda_1}_\om$. As $\exp^{\Lambda_2}_\om(B^{N_{x_0}}(0,\epsilon))$ and $\Lambda_1$ intersect cleanly at $x_0$ only
and following the proof of \cref{theointer}, for $\epsilon>0$
small enough we can extend the union $\exp^{\Lambda_2}_\om\cup\ \phi_{x_0}^{\Lambda_1}:B^{N_{x_0}}(0,\epsilon)\cup B^{\R^n}(0,\epsilon)\fl X$ to a diffeomorphism $\phi_{x_0}:B^{\R^{2n}}(0,\epsilon)\fl V$ as in \cref{defphi}, identifying $U$ with $B^{\Sigma}(0,\epsilon)$, where $\Sigma$ is an isotropic subspace of $(\R^{2n},\Om)$ and where the fibre $(N_{x_0},g^{N}_{\om,x_0})$ is seen as an Euclidean subspace of $(\R^{2n},\<\cdot,\cdot\>)$.

Let us identify $E,L$ over $B^{\R^{2n}}(0,\epsilon)$ with $E_{x_0},L_{x_0}$ as in \cref{locmod} and use $\zeta_1(x_0)$ to identify $L_{x_0}$ with $\C$. Then $\zeta_1,\zeta_2$ are identified with $1,\overline{\lambda}\in\C$ over $B^{\R^{2n}}(0,\epsilon)$, where $\lambda=\<\zeta_1(x_0),\zeta_2(x_0)\>_L$. Let $du$ be the Lebesgue measure of $\Sigma$ and $h_1\in\cinf(B^{\Sigma}(0,\epsilon),\R)$ be such that for $u\in B^{\Sigma}(0,\epsilon)$,
\begin{equation}\label{volE2}
dv_{\Lambda_1}(u)=h_1(u)du,~~\text{with}~~
h_2(0)=(dv_{\Lambda_1}/dv_{\Lambda_1,\om})(x_0).
\end{equation}
By \cref{thetagal} and \cref{asy}, for any $k\in\N$ and $\delta\in\ ]0,1[$, we get $\theta\in\ ]0,1[$ such that as $p\fl+\infty$,
\begin{multline}\label{loclagY2'}
\int_{B^{N_{x_0}}(0,\epsilon)}I(f_1,f_2)(x_0,w)dw\\
 =\int_{B^{N_{x_0}}(0,\epsilon)}\int_{B^{\Sigma}(0,\epsilon)}\<P_p(w,u)\zeta_1^pf_1(u),\zeta_2^pf_2(w)\>_{E_p} h_2(x_0,w)dv_{\Lambda_1}(u)dw\\
=\int_{B^{N_{x_0}}(0,\epsilon p^{-\theta/2})}\int_{B^{\Sigma}(0,\epsilon p^{-\theta/2})}\<P_p(w,u)\zeta_1^pf_1(u),\zeta_2^pf_2(w)\>_{E_p}\\
h_2(x_0,w)h_1(u)dudw+O(p^{-\infty})\\
=  \lambda^p p^n\sum_{r=0}^k p^{-\frac{r}{2}}\int_{B^{N_{x_0}}(0,\epsilon p^{-\theta/2})}\int_{B^{\Sigma}(0,\epsilon p^{-\theta/2})}\\
\<J_{r,x_0}\PP_{x_0}(\sqrt{p}w,\sqrt{p}u)f_{1,x_0}(u),f_{2,x_0}(w)\>_E\\
 \kappa_{x_0}^{-1/2}(w)\kappa_{x_0}^{-1/2}(u)h_2(x_0,w)h_1(u)dudw+p^{n-\frac{d_1+d_2}{2}+\frac{l}{2}}O(p^{-\frac{k+1}{2}+\delta}).
\end{multline}
Consider now the Taylor expansions up to order $k\in\N$ of $h_j\kappa_{x_0}^{-1/2}f_{j,x_0}$ for $j=1,2$ as in \cref{Taylorfk}. As in the proof of \cref{theointer}, we get $\delta'>0$ and a sequence $\{F_{x_0,r}\}_{r\in\N}$ of polynomials in two variables of $\R^{2n}$ with values in $\C$, of the same parity as $r$ and with
\begin{equation}
F_{x_0,0}(Z,Z')=\<f_1(x_0),f_2(x_0)\>_E\frac{dv_{\Lambda_1}}
{dv_{\Lambda_1,\om}}\frac{dv_{\Lambda_2}}{dv_{\Lambda_2,\om}}(x_0)
\det\left(\dot{R}^L_{x_0}/2\pi\right),
\end{equation}
such that as $p\fl+\infty$, equation \cref{loclag3} becomes
\begin{multline}\label{loclagY4}
\int_{B^{N_{x_0}}(0,\epsilon)}I(f_1,f_2)(x_0,w)dw=p^{n-\frac{d_1+d_2}{2}+\frac{l}{2}}\lambda^p\\
\sum_{r=0}^k p^{-r/2}\int_{N_{x_0}}\int_{\Sigma}F_{x_0,r}(w,u)\PP_{x_0}(w,u)dudw+p^{n-\frac{d_1+d_2}{2}+\frac{l}{2}}O(p^{-\frac{k+1}{2}+\delta'}).
\end{multline}
Thus writing
\begin{equation}\label{brx0}
b_r(x_0)=\int_{N_{x_0}}\int_{\Sigma}F_{x_0,r}(w,u)\PP_{x_0}(w,u)dudw,
\end{equation}
and recalling that the estimates are uniform in $x_0\in Y$, we get from \cref{midint}, \cref{intmidint} and \cref{loclagY4},
\begin{equation}\label{u1u2sumgal}
\begin{split}
& \<s_{1,p},s_{2,p}\>_p=p^{n-\frac{d_1+d_2}{2}+\frac{l}{2}}\lambda^p\sum_{r=0}^k p^{-r/2}\int_Y b_r(x)dv_Y(x)+p^{n-\frac{d_1+d_2}{2}+\frac{l}{2}}O(p^{-\frac{k+1}{2}}).
\end{split}
\end{equation}
Now, we can use \cref{comput1} to compute \cref{brx0} in general, and the argument of parity holds in the same way here, so that the coefficients $b_r$ defined in \cref{brx0} for $r\in\N$ vanish identically for $r$ odd. By \cref{u1u2sumgal}, this gives \cref{<u1u2>}.

Assume now $\dim\Lambda_1=n$, and let us compute
\begin{multline}\label{b0N}
b_0(x_0)=\frac{dv_{\Lambda_1}}{dv_{\Lambda_1,\om}}
\frac{dv_{\Lambda_2}}{dv_{\Lambda_2,\om}}(x_0)
\det\left(\dot{R}^L_{x_0}/2\pi\right)
\<f_1(x_0),f_2(x_0)\>_E\\
\int_{N_{x_0}}\int_{\Sigma}\PP_{x_0}(w,u)dudw.
\end{multline} 
In the same way than in the proof of \cref{theointer}, we can take the canonical symplectic basis $\{e_j,f_j\}_{j=1}^n$ of $(\R^{2n},\Om)$ such that $\Sigma=\R^n\times\{0\}$ and such that $\<e_{n-l+1},\dots, e_n\>$ is an oriented orthonormal basis of $(T_{x_0}Y,g^{TY}_\om)$ in the identification of $\R^{2n}$ with $T_{x_0}X$ via $d\phi_{x_0}$. Let $\nu_1,\dots, \nu_{d_2-l}\in N_{x_0}$ be such that $\<\nu_1,\dots, \nu_{d_2-l},e_{n-l+1},\dots, e_n\>$ is an oriented orthonormal basis of the isotropic subspace $\Sigma_2:=N_{x_0}\oplus T_{x_0}Y$. Then for $ 1\leq i\leq d_2-l$ and $n-l+1\leq j\leq n$, we have that $\<\nu_i,f_j\>=-\omega(\nu_i,e_j)=0$. Thus setting
\begin{equation}\label{ABdefgal}
\begin{split}
A=(a_i^j)_{1\leq i\leq n-l,\,1\leq j\leq d_2-l}\quad & \text{with}\quad a_i^j=\om(e_i,\nu_j),\\
B=(b_i^j)_{1\leq i\leq n-l,\,1\leq j\leq d_2-l}\quad & \text{with}\quad b_i^j=\<e_i,\nu_j\>,
\end{split}
\end{equation}
we get for all $1\leq j\leq d_2-l$,
\begin{equation}\label{viejfjgal}
\nu_j=\sum_{i=1}^{n-l} b_i^je_i+\sum_{i=1}^{n-l} a_i^jf_i.
\end{equation}
Write $dt:=dt_1\dots dt_{d_2-l}$ for the Lebesgue measure of $\R^{d_2-l}$. Using the summation convention of \cref{subsetting} with $i$ from $1$ to $n-l$ and $j,k$ from $1$ to $d_2-l$ whenever they appear as free indices, we get  
\begin{equation}\label{comput1gal}
\begin{split}
\int_{N_{x_0}}&\int_{\Sigma}\PP_{x_0}(w,u)dudw=\int_{\R^{d_2-l}}\int_{\R^n}\PP_{x_0}(t_j\nu_j,u_ie_i)dudt\\
&=\int_{\R^{d_2-l}}\int_{\R^n} \exp\left(-\frac{\pi}{2}\sum\limits_{i=n-l+1}^n u_i^2\right)\\
&\quad\quad\quad\quad\quad\quad\exp\left(-\frac{\pi}{2}\sum\limits_{i=1}^{n-l} \left((u_i-t_jb_i^j)^2+(t_ja_i^j)^2+2\sqrt{-1}u_i t_ja_i^j\right)\right)du dt\\
&=2^{l/2}\int_{\R^{d_2-l}}\int_{\R^{n-l}} \exp\Big(-\frac{\pi}{2}\sum\limits_{i=1}^{n-l}u_i^2+(t_ja_i^j)^2\\
&\quad\quad\quad\quad\quad\quad +2\sqrt{-1}u_i t_ja_i^j+2\sqrt{-1} t_kb_i^ka_i^jt_j\Big)du_1\dots du_{n-l} dt\\
&=2^{n/2}\int_{\R^{d_2-l}} \exp\left(-\frac{\pi}{2}\sum\limits_{j=1}^{n-l} \left(2(t_ja^j_i)^2+2\sqrt{-1} t_kb_i^ka_i^jt_j\right)\right)dt.
\end{split}
\end{equation}
As $\om(\nu_j,\nu_k)=0$, for all $1\leq j,k\leq d_2-l$, we know from \cref{viejfjgal} that the matrix $B^TA$ is symmetric. Then as in \cref{comput2bis}, we get
\begin{equation}\label{comput2gal}
\int_{N_{x_0}}\int_{\Sigma}\PP_{x_0}(w,u)dudw=2^{n/2}\det{}^{-1/2}(\sqrt{-1}A(B-\sqrt{-1}A)).
\end{equation}
Using the explicit definition of $A$ and $B$ above and from \cref{comput2gal,HermTX,dv=dvom}, we get \cref{bm0}.

\end{proof}

\begin{rem}\label{meta}
Suppose that the first Chern class $c_1(TX)$ of $(TX,J)$ is even in $H^2(X,\Z)$. Then there exists a complex line bundle $K_X^{1/2}$ over $X$ such that its second tensor power is equal to the canonical line bundle $K_X$ of $X$. The choice of $K_X^{1/2}$ does not depend on $J$ compatible with $\om$, and is called a \emph{metaplectic structure} on $(X,\om)$. Note that such a choice is not unique in general.
Now if $\iota:\Lambda\fl X$ is an immersed Lagrangian submanifold, then $\iota^*K_X$ is canonically isomorphic to $\det(T^*\Lambda_\C)$ over $\Lambda$, and we call $\iota^*K_X^{1/2}$ the \emph{half-form bundle} of $\Lambda$. We endow $K_X^{1/2}$ with the Hermitian structure induced by $h^{K_X}_\om$ as in \cref{subsetting}.

Consider now the setting of \cref{theointergal}, with $\dim\Lambda_1=\dim\Lambda_2=n$ and $g^{TX}=g^{TX}_\om$. Via the isomorphism above, we define the \emph{angle} of $\iota_j:\Lambda_j\fl X,~j=1,2$, as a function on any connected component $Y$ of their intersection by the formula
\begin{equation}\label{detlambda}
\begin{split}
\det\{\Lambda_1,\Lambda_2\}&=h^{K_X}_\om(dv_{\Lambda_1},dv_{\Lambda_2})^{-1}\\
&=\det\left\{h^{TX}_\om(e_i,\nu_j)\right\}_{i,j=1}^{n-l}.
\end{split}
\end{equation}
On the other hand, following \cite[Lemma 3.1]{BPU95}, we can construct a sesquilinear pairing $\#:\iota_1^*K^{1/2}_{X}|_Y\,\times\,\iota_2^*K_{X}^{1/2}|_Y\fl\det(T^*Y_\C)$ over $Y$, depending only on the metaplectic structure of $(X,\om)$, which at any $x\in Y$ takes two square roots $dv_{\Lambda_j,x}^{1/2}$ of $dv_{\Lambda_j,x}$ for $j=1,2$ to
\begin{equation}\label{etadef}
dv_{\Lambda_1,x}^{1/2}\#dv_{\Lambda_2,x}^{1/2}=\det{}^{-1/2}\{\om(e_i,\nu_j)\}_{i,j=1}^{n-l}dv_{Y,x},
\end{equation}
for an Euclidean volume form $dv_{Y,x}$ of $(T_xY,g^{TY}_{x})$ and some coherent choice of square root induced by $dv_{\Lambda_1,x}^{1/2},~dv_{\Lambda_2,x}^{1/2}$ and $dv_{Y,x}$. Then taking $E=K_X^{1/2}$, \cref{theointergal} gives the following formula for $b_{0}$ on $Y$ as in \cref{bm0},
\begin{equation}\label{bm0BPU}
b_{0}=2^{\frac{n-l}{2}}e^{-\sqrt{-1}\frac{\pi(n-l)}{2}}\int_{Y} \det\{\Lambda_1,\Lambda_2\}^{-1}f_1\#f_2.
\end{equation}
In the particular case of $(X,J,\om)$ Kähler,
this formula can be compared with the one appearing in
\cite[Prop.\ 3.16]{BPU95}. In particular, they get $\det\{\Lambda_1,\Lambda_2\}^{-1/2}$ instead of $\det\{\Lambda_1,\Lambda_2\}^{-1}$ as in \cref{bm0BPU}. This discrepancy is due to the fact that even though they use half-forms, their Lagrangian states take values in $L^p$ and not in $L^p\otimes K_X^{1/2}$ as it is the case here.
Note that without metaplectic structure on $(X,\om)$, only the product of the square root of \cref{detlambda} with \cref{etadef} makes sense in general (see \cite{Tuy16} for related results).
\end{rem}

Finally, note that the assumption $\dim\Lambda_1=n$ for formula
\cref{bm0} of the first coefficient of the expansion
was used in the proof of \cref{theointergal} to compute the
elegant formula \cref{comput2gal}. Without this assumption,
one still gets an integral of the form \cref{comput1gal} by
following the method of the proof, and the classical formulas of
\cref{Gausssec} for the
Gaussian integral can be used to compute it explicitly.
One also gets a formula in terms of the symplectic
form, the Riemannian metric and local frames via the definition
\cref{ABdefgal} of the
coefficients appearing inside the Gaussian function.

\section{Extensions to non-compact manifolds and orbifolds}\label{secnoncpctorbi}

In this section, we show how one can adapt the results of the previous Sections in the case of non-compact manifolds and orbifolds. We will work for simplicity in the case of $(X,J,\om)$ Kähler and $g^{TX}=g^{TX}_\om$. Then as underlined in the introduction, the renormalized Bochner Laplacian \cref{deltalpe} reduces to the Kodaira Laplacian on sections. 

Note further that the existence of an expansion of the form \cref{Toepasyexp} is a straightforward consequence of the existence of an expansion as in \cite[(4.9)]{MM08b}.

\subsection{Non-compact case}\label{secnoncpct}

Let $(X,J,\om,g^{TX})$ be a complete Kähler manifold with $\om(\cdot,\cdot)=g^{TX}(J\cdot,\cdot)$, let $(L,h^L)$ be a holomorphic line Hermitian bundle over $X$ with Chern connection $\nabla^L$ satisfying \cref{preq}, and let $(E,h^E)$ be an auxiliary holomorphic Hermitian bundle with Chern connection $\nabla^E$. For any $p\in\N^*$, let $H^0_{(2)}(X,E_p)$ denote the space of holomorphic sections of $E_p=L^p\otimes E$ which are square integrable with respect to the $L^2$-Hermitian product defined as in \cref{L2}. Let $P_p$ denote the orthogonal projection from the space of $L^2$-sections of $E_p$ onto $H^0_{(2)}(X,E_p)$ with respect to this product. Then as noticed in \cite[Rem.1.4.3]{MM07}, $P_p$ has smooth Schwartz kernel $P_p(\cdot,\cdot)\in\cinf(X\times X,E_p\boxtimes E_p^*)$ with respect to the Riemannian volume form $dv_X$ of $(X,g^{TX})$, and $P_p(\cdot,\cdot)$ is square integrable and holomorphic with respect to its first variable.

Let us write $R^{det}$ for the curvature of the Chern connection of $K_X^*$. Then we have the following result.

\begin{theorem}{\cite[Th.~5.2, 5.3]{MM08b}}\label{noncompact}
Suppose that there exists $C>0$ such that  for all $x\in X$ and $v\in T_xX$, the following inequality holds in the sense of endomorphisms of $E$,
\begin{equation}\label{hypL2}
\sqrt{-1}(R^{det}\Id_E+R^E)(v,Jv)>-C\om(v,Jv)\Id_E.
\end{equation}

Then for any compact set $K\subset X$, \cref{theta} holds uniformly for any $x,x'\in K$ and \cref{asy} holds uniformly for any $x_0\in X$.

If $F\in\cinf(X,\End(E))$ has compact support, then \cref{Toepasy} holds uniformly for any $x_0\in X$.
\end{theorem}

From now on, we suppose that \cref{hypL2} is verified for $X$. Then \cref{BS} still makes sense in this context, provided $\Lambda$ is compact. Precisely, for $(\Lambda,\iota,\zeta)$ Bohr-Sommerfeld manifold as in \cref{BS} with $\Lambda$ compact and for $f\in\cinf(\Lambda,\iota^*E)$, we define the associated isotropic state $\{s_{f,p}\}_{p\in\N}$ in the same way than in \cref{defLagstate} for any $p\in\N^*$ and $x\in X$ by the formula 
\begin{equation}\label{defLagstateL2}
s_{f,p}(x)=\int_{\Lambda} P_p(x,\iota(y))\iota_p.\zeta^pf(y)dv_{\Lambda}(y).
\end{equation}

Then as $\Lambda$ is compact, we get that $s_{f,p}\in H^0_{(2)}(X,E_p)$. Furthermore, the following analogue of \cref{proprepgal} holds.
\begin{lem}\label{deflagL2}
Suppose that $(X,J,\om,g^{TX})$ is a complete Kähler manifold satisfying \cref{hypL2}, and let $(\Lambda,\iota,\zeta)$ be a compact Bohr-Sommerfeld submanifold of $X$. Then for any $s\in H^0_{(2)}(X,E_p)$, the following reproducing property holds,
\begin{equation}\label{repL2}
\<s,s_{f,p}\>_p=\int_{\Lambda} \<s(\iota(x)),\iota_p.\zeta^pf(x)\>_{E_p}\ dv_{\Lambda}(y).
\end{equation}

Furthermore, for any $F\in\cinf(X,\End(E))$ with compact support, property \cref{Toeplag} holds.
\end{lem}
\begin{proof}
As $\Lambda$ is compact, we can repeat the computations of \cref{comprep}, so that \cref{repL2} holds. As $F\in\cinf(X,\End(E))$ has compact support, we can repeat in the same way the computations of \cref{compToeplag}, and \cref{Toeplag} holds as well in this context.
\end{proof}

With these preliminaries, we can state the following generalization of the results of \cref{secasylagstate}, \cref{sectrans} and \cref{secclean}.

\begin{theorem}\label{thL2}
Suppose that $(X,J,\om,g^{TX})$ is a complete Kähler manifold satisfying \cref{hypL2}. If $(\Lambda,\iota,\zeta)$ is a compact Bohr-Sommerfeld submanifold of $(X,\om)$, then \cref{theonorme} holds.

Furthermore, if $(\Lambda_j,\iota_j,\zeta_j),~j=1,2$, are two compact Bohr-Sommerfeld submanifolds of $(X,\om)$ intersecting cleanly, then \cref{theointergal} hold.
\end{theorem}
\begin{proof}
Let $(\Lambda_j,\iota_j,\zeta_j),~j=1,2$, be two compact Bohr-Sommerfeld submanifolds of $X$, and consider $f_j\in\cinf(X,\iota_j^*E),~j=1,2$. By \cref{noncompact}, we know that \cref{uopinf} is still true uniformly in any compact set $K\subset X$. Furthermore, using \cref{defLagstateL2}, \cref{repL2} and omitting the immersions, we get for any $p\in\N^*$,
\begin{equation}\label{repnormeL2}
\begin{split}
\<s_{f_1,p},s_{f_2,p}\>_p & =\int_{\Lambda_2}\<s_{f_1,p}(x),\zeta_2^pf_2(x)\>_{E_p}dv_{\Lambda_2}(x)\\
& =\int_{\Lambda_2}\int_{\Lambda_1}\<P_p(x,y)\zeta_1^pf_1(y),\zeta_2^pf_2(x)\>_{E_p}dv_{\Lambda_1}(y)dv_{\Lambda_2}(x).
\end{split}
\end{equation}
We can then choose the compact set $K$ in \cref{noncompact} to contain $\iota(\Lambda_1)\cup\iota(\Lambda_2)$, and the proof of \cref{thL2} goes along the lines of the proofs of \cref{theonorme}, \cref{theointer} and \cref{theointergal}. By the second part of \cref{deflagL2}, the case of $\<T_{F,p}s_{f_1,p},s_{f_2,p}\>_p$ such that $F\in\cinf(X,\End(E))$ has compact support is strictly analogous.
\end{proof}

\subsection{Orbifold case}\label{secorbi}

In this section, we consider a complete Kähler orbifold
$(X,J,\om,g^{TX})$ satisfying \cref{hypL2}, a proper holomorphic
Hermitian orbifold line $(L,h^L)$ bundle over $X$ with Chern
connection $\nabla^L$ satisfying \cref{preq}, and a proper
holomorphic Hermitian orbifold vector bundle $(E,h^E)$ over $X$
endowed with its Chern connection $\nabla^E$. In order to give a precise meaning to these notions, we first state some notations and definitions from \cite[\S\ 5.4]{MM07}.

\begin{defi}\label{M}
Let $\cal{M}$ be the category whose objects are the pairs $(M,G)$, with $M$ smooth connected manifold and $G$ a finite group acting effectively on $M$, and whose morphisms $\Phi:(M,G)\fl(M',G')$ are families of open embeddings $\varphi:M\fl M'$ satisfying:
\begin{itemize}
\item For each $\varphi\in\Phi$, there is an injective group homomorphism $\lambda_\varphi:G\fl G'$ such that $\varphi$ is $\lambda_\varphi$-equivariant.
\item For $g\in G'$ and $\varphi\in\Phi$, define $g\varphi:M\fl M'$ by $(g\varphi)(x)=g\varphi(x)$ for any $x\in M$. If $(g\varphi)(M)\cap\varphi(M)\neq\0$, then $g\in\lambda_\varphi(G)$.
\item For $\varphi\in\Phi$, we have $\Phi=\{g\varphi~|~g\in G'\}$.
\end{itemize}
\end{defi}

\begin{defi}\label{orbistructdef}
Let $X$ be a paracompact Hausdorff space and let $\cal{U}_X$ be a covering of $X$ consisting of connected open subsets, satisfying the condition
\begin{equation}\label{filter}
\begin{split}
&\text{For any $U,U'\in\cal{U}_X$ and $x\in U\cap U'$,}\\
&\text{there is $U''\in\cal{U}_X$ such that $x\in U''\subset U\cap U'$.}
\end{split}
\end{equation}
An \emph{orbifold structure} $\cal{V}_X$ on $X$ consists of the following datas:

\begin{itemize}
\item For any $U\in\cal{U}_X$, an object $(G_U,\til{U})$ of $\cal{M}$ and a ramified covering $\tau_U:\widetilde{U}\fl U$ which is $G_U$-invariant and induces a homeomorphism $U\simeq\ \til{U}/G_U$.
\item  For any $U,~V\in\cal{U}_X$ such that $U\subset V$, a morphism $\Phi_{VU}:(G_U,\til{U})\fl (G_V,\til{V})$ of $\cal{M}$, which covers the inclusion $U\subset V$ and satisfies $\Phi_{WU}=\Phi_{WV}\circ\Phi_{VU}$ for any $U,~V,~W\in\cal{U}_X$, with $U\subset V\subset W$.
\end{itemize}

If $\cal{U}_X'$ is a refinement of $\cal{U}_X$ satisfying the condition \cref{filter}, then there is an orbifold structure $\cal{V}_X'$ associated with $\cal{U}_X'$ such that $\cal{V}_X\cup\cal{V}_X'$ is again an orbifold structure. We then say that $\cal{V}_X$ and $\cal{V}_X'$ are \emph{equivalent}. An equivalence class is called an \emph{orbifold structure} on $X$. In particular, we can suppose that $\cal{U}_X$ is arbitrarily fine. In the sequel, we will always consider the unique maximal representative in the equivalence class. 
\end{defi}

In the above definitions, we can replace the objects of $\cal{M}$ by manifolds with specified structures together with a group preserving these structures, and morphisms preserving these structures. In the case in hand, by structure we mean an orientation, a Riemannian metric, a symplectic structure, an almost-complex structure or a complex structure. Furthermore, we can realise Cartesian products of orbifolds in the obvious way.

Let $(X,\cal{V}_X)$ be an orbifold. For each $x\in X$, up to refinement of $\cal{V}_X$, there exists $U_x\in\cal{U}_X$ containing $x$ and $\til{x}\in\til{U}$, $\tau_U(\til{x})=x$, such that $\til{x}$ is a fixed point of $G_U$. Then by the second axiom of \cref{M}, such a group is unique up to isomorphism, and we denote it by $G^X_x$. If $|G_x^X|=1$, then $X$ has a smooth structure in a neighbourhood of $x$, and we call such an $x$ a \emph{smooth point} of $X$. If $|G_x^X|>1$, we call such an $x$ a \emph{singular point} of $X$. We denote $X_{sing}=\{x\in X~|~|G_x^X|>1\}$ the \emph{singular set} of $X$, and $X_{reg}=\{x\in X~|~|G_x^X|=1\}$ the \emph{regular set} of $X$. In the sequel, we always denote by $\til{x}\in\til{U}$ a lift of $x\in U\in\UU_X$. 

%

The next definitions are adaptations of the notions of
orbifold embedding and submersion from \cite[Def.\ 1.6,\ Def.\ 1.7]{Ma05}.

\begin{defi}\label{orbimap}
An \emph{orbifold immersion} $I:(Y,\cal{V}_Y)\fl(X,\cal{V}_X)$ is a continuous map $\iota:Y\fl X$, such that for any $V\in\UU_X$ and any $U\in\UU_Y$ connected component of $\iota^{-1}(V)$, there is a family $I_{UV}$ of immersions $\iota_{UV}:\til{U}\fl\til{V}$ covering $\iota$ together with surjective group homomorphisms $\lambda_{UV}:G_V\fl G_U$ such that $\iota_{UV}$ is $\lambda_{UV}$-equivariant. Furthermore, the families $I_{UV}$ satisfy $I_{UV}=\{g\iota_{UV}~|~g\in G_U\}$ and are compatible with the orbifold structures in the obvious sense.
In that case, we define the \emph{stabilizer} of $V$ in $U$ by $K_{UV}=\Ker\lambda_{UV}$. Then $m_{X,Y}:=|K_{UV}|$ is locally constant on $Y$, and is called the \emph{relative multiplicity} on $Y$.

A \emph{singular immersion} $\hat{I}$ from a smooth manifold $Y$ to an orbifold $(X,\cal{V}_X)$ is a continuous map $\iota:Y\fl X$, together with immersions $\til{\iota}_V:U\fl\til{V}$ covering $\iota$ for any $V\in\cal{U}_X$, such that $g.\iota(U)$ intersects $\iota(U)$ cleanly in the sense of \cref{cleandef} for all $g\in G_V$. In that case, we define the stabilizer of $U$ in $V$ by the subgroup $K_{UV}\subset G_V$ fixing each point of $\til\iota_V(U)$. Then the \emph{relative multiplicity} $m_{X,Y}=|K_{UV}|$ is again locally constant on $Y$.

An \emph{orbifold submersion} $P:(M,\cal{V}_M)\fl(X,\cal{V}_X)$ is a continuous map $\pi:M\fl X$ such that $\pi(U)\in\cal{U}_X$ for any $U\in\cal{U}_M$, together with submersions $\pi_U:\til{U}\fl\til{\pi(U)}$ covering $\pi$ and surjective group homomorphisms $\lambda_U:G_U\fl G_{\pi(U)}$ for any $U\in\cal{U}_X$ making $\pi_U$ be $\lambda_U$-equivariant. Furthermore, we assume compatibility with the orbifold structures in the obvious sense.
\end{defi}

Note that any $x\in X$ can be seen as an immersed orbifold with $m_{X,x}=|G_x|$. In both definitions of an immersion above, if $\iota^{-1}(X_{sing})$ has strictly positive measure for the density induced by any Riemannian metric, then $G_V$ fixes $\iota(U)$ and $m_{X,Y}$ is strictly positive. The intersection of two orbifold immersions is still defined as in \cref{cleandef} to be their fibred product over $X$, which gets a natural orbifold structure making all maps into orbifold immersions.

Finally, note that we can easily combine the definitions above to get the notion of a singular orbifold immersion, and the results of this section hold in this case as well. For simplicity and clarity, we will keep both notions separated from each other.

\begin{defi}
An \emph{orbifold vector bundle} is an orbifold submersion $P:(E,\cal{V}_E)\fl(X,\cal{V}_X)$ such that $E_U:=\pi^{-1}(U)$ belongs to $\cal{U}_E$ for any $U\in\cal{U}_X$ and $\pi_{E_U}:\til{E}_U\fl\til{U}$ are $G_{E_U}$-equivariant vector bundles. Furthermore, we ask the inclusions $\Phi_{E_VE_U}$ covering $\Phi_{VU}$ to be equivariant
vector bundle maps, for any $U,~V\in\cal{U}_X$ such that $U\subset V$.

If $G_{E_U}$ acts effectively on $\til{U}$ for all $U\in\cal{U}_X$, that is the group morphisms $\lambda_{E_U}:G_{E_U}\fl G_U$ associated with $P$ as in \cref{orbimap} are isomorphisms, we say that $E$ is \emph{proper}.
\end{defi}

We can then define the proper tangent orbifold bundle $TX$ and the proper cotangent orbifold bundle $T^*X$ over any orbifold $(X,\cal{V}_X)$ in the obvious way. We can as well form tensor products of vector bundles by taking the tensor products locally over each orbifold chart, and we check easily that this operation preserves properness. If $E$ is a proper orbifold bundle over $X$ and if $\Psi:(X,\cal{V}_X)\fl(Y,\cal{V}_Y)$ is any of the orbifold maps of \cref{orbimap}, we can pullback $E$ to $Y$ by $\Psi$ in the obvious way, and we write $\Psi^*E$ for the pullback orbifold vector bundle, which is still proper.

We define a distance on $X$ for any $x, y\in X$ by
\begin{multline}
d(x,y)= \inf_{\gamma}\Big\{\sum_j\int_{t_{j-1}}^{t_{j+1}} |\dt\til{\gamma}_j(t)| dt~\Big|~\gamma:[0,1]\fl X,~\gamma(0)=x,~\gamma(1)=y,\\
\text{such that there exists}\ t_0=0<t_1<\dots<t_k=1,~\gamma([t_{j-1},t_j])\subset U_j,\\
U_j\in\cal{U}_X,~\text{and a smooth map}\ \til{\gamma}_j:[t_{j-1},t_j]\fl\til{U}_j\ \text{that covers}\ \gamma|_{[t_{j-1},t_j]} \Big\}.
\end{multline}
Let $E\fl X$ be an orbifold vector bundle. An orbifold section $s:X\fl E$ is called \emph{smooth} if for each $U\in\cal{U}_X$, the restriction of $s$ to $U$ is covered by a $G^E_U$-equivariant smooth section $\til{s}_U:\til{U}\fl\til{E}_U$. In the same way, if $X$ is a complex orbifold and $E$ is a holomorphic orbifold vector bundle, we say $s$ is holomorphic if it is locally covered by holomorphic sections. The space of smooth (resp. holomorphic) sections of $E$ is denoted by $\cinf(X,E)$ (resp. $H^0(X,E)$).

If $X$ is oriented and $\alpha$ is a smooth section of the exterior product orbifold bundle $\Lambda(T^*X)$ with support in $U\in\cal{U}$, we define
\begin{equation}\label{orbiint}
\int_X\alpha=\frac{1}{|G_U|}\int_{\til{U}}\til{\alpha}_U,
\end{equation}
where $\til{\alpha}_U$ is an invariant section covering $\alpha$ over $\til{U}$. We extend this definition for general $\alpha$ using a partition of unity. In particular, if $X$ is oriented and Riemannian, there is an induced Riemannian volume form $dv_X$ on $X$, so that we can integrate functions.

Let now $(X,J,\om)$ be a Kähler orbifold. As we can verify locally, for any Hermitian holomorphic proper orbifold bundle over $X$, its Chern connection is well-defined and unique. Let $(L,h^L)$ be a
proper holomorphic Hermitian orbifold line bundle over $X$,
such that its
Chern connection satisfies \cref{preq}. We write $g^{TX}$ for the Riemannian metric on $X$ satisfying \cref{Jgras}, and $dv_X$ for the associated Riemannian volume form. Let $(E,h^E)$ be an auxiliary
proper holomorphic Hermitian orbifold vector bundle over $X$.

We define the $L^2$-Hermitian product associated with all the previous datas on $\cinf(X,E_p)$ by the formula \cref{L2}, and the \emph{Bergman kernel} $P_p(\cdot,\cdot)\in\cinf(X\times X,E_p\boxtimes E_p^*)$ is the Schwartz kernel with respect to $dv_X$ of the orthogonal projection $P_p$ from $\cinf(X,E_p)$ to $H^0_{(2)}(X,E_p)$ as in \cref{ker}. For any $V\in\cal{U}_X$ and all $p\in\N^*$, let $\til{P}_p(\cdot,\cdot)\in\cinf(\til{V}\times\til{V},\til{E}_{p,V}\boxtimes\til{E}_{p,V}^*)$ be the $G_V\times G_V$-invariant lift of $P_p(\cdot,\cdot)\in\cinf(V\times V,E_p\boxtimes E_p^*)$. More generally, for any object on $V\in\cal{U}_X$, we add a superscript $~\til{}~$ to denote the corresponding object on $\til{V}$.

For any $m\in\N$, let $|\cdot|_{\CC^m}$ denote the local
$\CC^m$-norm on local sections of $E_p\boxtimes E_p^*$ over $X\times X$ induced by $h^L,~h^E$ and $\nabla^L,~\nabla^E$. The following result is the version of \cref{asy} for orbifolds. It uses the fact, noticed in \cite{Ma05}, that the finite propagation speed of the wave equation holds on orbifolds.

\begin{prop}{\cite[\S~6.2]{MM08b},\cite[Rem.\ 5.4.12b)]{MM07}}\label{fdpprop}
\cref{theta} holds in the case of $(X,J,\om,g^{TX})$ complete Kähler orbifold satisfying \cref{hypL2}. Moreover, for any $V\in\cal{U}_X$, there exists a section $F(\til{D}_p)(\cdot,\cdot)\in\cinf(\til{V}\times\til{V},\til{E}_{p,V}\boxtimes\til{E}_{p,V}^*)$ satisfying the following properties:

For any $\til{x},~\til{y}\in\til{V}$ and $g\in G_V$,
\begin{equation}
(g,1)F(\til{D}_p)(g^{-1}\til{x},\til{y})=(1,g^{-1})F(\til{D}_p)(\til{x},g\til{y}).
\end{equation}

For any $m,~l\in\N$, there is $C_{m,l}>0$ such that for any $\til{x},~\til{y}\in\til{V}$ and all $p\in\N^*$,
\begin{equation}\label{fdp}
|\til{P}_p(\til{x},\til{y})-\sum_{g\in G_U}(1,g^{-1})F(\til{D}_p)(\til{x},g\til{y})|_{\CC^m}\leq C_{m,l} p^{-l}.
\end{equation}

$F(\til{D}_p)(\cdot,\cdot)$ satisfies the expansion of \cref{asy} at any $x_0\in\til V$.

\end{prop}

With all these prerequisites in hand, \cref{BS} still makes sense in this context replacing the immersion $\iota$ by an orbifold immersion or singular immersion $I$ as in \cref{orbimap}. In the second case, we talk about a \emph{singular} Bohr-Sommerfeld submanifold. In any case, if $\Lambda$ is compact, the associated isotropic state as in \cref{defLagstate} is well defined and \cref{proprepgal} still holds. We will use the additivity property \cref{add} to assume that the section $f$ of \cref{Lagstate} has compact support in some given open set $U\in\cal{U}_{\Lambda}$.

\begin{theorem}\label{orbith}
Let $(X,J,\om,g^{TX})$ be a complete Kähler orbifold satisfying \cref{Jgras}, let $(L,h^L)$ be a holomorphic Hermitian proper orbifold line bundle such that the curvature of its Chern connection satisfies \cref{preq}, and let $(E,h^E)$ be a holomorphic Hermitian proper orbifold vector bundle. Suppose that $(X,J,\om,g^{TX})$ satisfies \cref{hypL2}.

If $(\Lambda,I,\zeta)$ is a compact Bohr-Sommerfeld submanifold of $X$ and if the endomorphism $F\in\cinf(X,\End(E))$ has compact support, then \cref{theonorme} holds, with the following formula for the first coefficient of \cref{Toepb0norme},
\begin{equation}\label{orbib0norme}
b_0=2^{d/2}m_{X,\Lambda}\int_{\Lambda}\<Ff,f\>_{\iota^*E}dv_{\Lambda}.
\end{equation}
%

If $(\Lambda_j,I_j,\zeta_j),~j=1,2$, are two compact Bohr-Sommerfeld submanifolds of $X$ intersecting cleanly and if $F\in\cinf(X,\End(E))$ has compact support, then the expansion of \cref{theointergal} holds. If $\dim\Lambda_1=n$, then the first coefficients $b_{q,0}$ of \cref{<u1u2>} satisfy the formula \cref{bmj0} multiplied by
\begin{equation}\label{orbibm0}
m_{X,\Lambda_2}/m_{\Lambda_1,Y_q}.
\end{equation}

Finally, the above holds for compact singular Bohr-Sommerfeld submanifolds of $X$, provided their intersection locus is away from the singular set.
\end{theorem}
\begin{proof}

Let $(\Lambda,I,\zeta)$ be a compact Bohr-Sommerfeld submanifold, and let $f\in\cinf(\Lambda,I^*E)$ have compact support in a sufficiently small open set $U\in\cal{U}_{\Lambda}$, connected component of $\iota^{-1}(V)$ for some $V\in\cal{U}_X$. Then using \cref{orbiint} and \cref{fdp}, for any $\til{x}\in\til{V}$ we have as $p\fl+\infty$,
\begin{equation}\label{orbirep}
\begin{split}
\til{s}_{f,p}(\til{x}) & =\frac{1}{|G_U|}\int_{\til{U}} \til{P}_p(\til{x},\iota_{UV}(\til{y}))\iota_{p,UV}.\til{f}\til{\zeta}^p(\til{y})dv_{\til{U}}(\til{y})\\
& =\frac{1}{|G_U|}\int_{\til{U}} \sum_{g\in G_V}(1,g^{-1})F(\til{D}_p)(\til{x},g\iota_{UV}(\til{y}))\iota_{p,UV}.\til{f}\til{\zeta}^p(\til{y})dv_{\til{U}}(\til{y})+O(p^{-\infty})\\
& =\frac{1}{|G_U|}\int_{\til{U}} \sum_{g\in G_V}F(\til{D}_p)(\til{x},\iota_{UV}(\til{y}))\iota_{p,UV}.(g.\til{f}\til{\zeta}^p(g^{-1}\til{y}))dv_{\til{U}}(\til{y})+O(p^{-\infty})\\
& =\frac{|G_V|}{|G_U|}\int_{\til{U}} F(\til{D}_p)(\til{x},\iota_{UV}(\til{y}))\iota_{p,UV}.\til{f}\til{\zeta}^p(\til{y})dv_{\til{U}}(\til{y})+O(p^{-\infty}).
\end{split}
\end{equation}
Here $\iota_{UV}:\til{U}\fl\til{V}$ is any member of the family of maps in $I_{UV}$. Now by \cref{orbimap}, we have $|G_V|/|G_U|=m_{X,\Lambda}$. By \cref{fdpprop}, $F(\til{D}_p)(\cdot,\cdot)$ satisfies the expansion of \cref{asy} at any $x_0\in\til V$, so that we can follow the proof of \cref{theonorme} to deduce from \cref{orbirep} an asymptotic expansion as $p\fl+\infty$ of the form \cref{norme} for the norm of $s_{f,p}$, with highest coefficient given by \cref{orbib0norme} in the case $F=\Id_E$.

For any $j=1,2$, let $(\Lambda_j,I_j,\zeta_j)$ be compact Bohr-Sommerfeld submanifolds and let $f_j\in\cinf(\Lambda,I^*E)$ have compact support in a sufficiently small open set $U_j\in\cal{U}_{\Lambda}$, connected component of $\iota^{-1}(V)$ for some $V\in\cal{U}_X$. Then as the reproducing property \cref{rep} still holds, analogous to \cref{loclagm1}, \cref{orbirep}, using \cref{orbiint}, \cref{fdp}, and omiting the immersion maps, we have
as $p\fl+\infty$,
\begin{multline}\label{loclagorbi}
\<s_{1,p},s_{2,p}\>_p =\frac{1}{|G_{U_2}|}\int_{\til{U}_2} \<\til{\zeta}_{f_1,p}(\til{x}),\til{\zeta}_2^p\til{f}_2(\til{x})\>_{E_p} dv_{\til{U}_2}(\til{x})\\
=\frac{1}{|G_{U_1}|}\frac{1}{|G_{U_2}|}\int_{\til{U}_2}\int_{\til{U}_1} \<\til{P}_p(\til{x},\til{y})\til{\zeta}_1^p\til{f}_1(\til{y}),\til{\zeta}_2^p\til{f}_2(\til{x})\>_{E_p}dv_{\til{U}_1}(\til{y})dv_{\til{U}_2}(\til{x})\\
=\frac{|G_V|}{|G_{U_1}||G_{U_2}|}\int_{\til{U}_2}\int_{\til{U}_1} \<F(\til{D}_p)(\til{x},\til{y})\til{\zeta}_1^p\til{f}_1(\til{y}),\til{\zeta}_2^p\til{f}_2(\til{x})\>_{E_p}dv_{\til{U}_1}(\til{y})dv_{\til{U}_2}(\til{x})\\
+O(p^{-\infty}).
\end{multline}
By \cref{orbimap}, we have $m_{X,\Lambda_2}=|G_V|/|G_{U_2}|$, and then $m_{\Lambda_1,y}=|G_y^{\Lambda_1}|=|G_{U_1}|$ for $U_1$ small enough. In the case of discrete intersection, we take $y\in\iota_2^{-1}(\iota_1(\Lambda_1)\cap\iota_2(\Lambda_2))$ and $V\in\UU_X$ to be a small enough neighbourhood of $\iota_1(y)\in X$ to get \cref{orbibm0} in the case $F=\Id_E$ and discrete intersection.

Recall \cref{cleandef}. Let now $\til{W}$ be the lift of some open set $W\in\UU_{Y}$, where is $Y$ the connected component of $\Lambda_1\cap\Lambda_2$ such that its image by $j_1$ intersects the support of $f_1$, and set $l=\dim Y$. In the case of clean intersection, we can follow the proof of \cref{theointergal} until \cref{u1u2sumgal} to get an asymptotic expansion of the form \cref{<u1u2>}, and get from \cref{loclagorbi} a sequence $b_r\in\cinf(Y,\C),~r\in\N$ such that
as $p\fl+\infty$,
\begin{equation}\label{orbiprodcomput}
\begin{split}
\<s_{1,p},s_{2,p}\>_p & =\frac{|G_V|}{|G_{U_1}||G_{U_2}|}p^{\frac{l}{2}}\lambda^p\sum_{r=0}^k p^{-r/2}\int_{\til{W}} \til{b}_r(\til{x})dv_{\til{W}}(\til{x})+O(p^{\frac{l}{2}-\frac{k+1}{2}})\\
& =\frac{|G_V|}{|G_{U_2}|}\frac{|G_W|}{|G_{U_1}|}p^{\frac{l}{2}}\lambda^p\sum_{r=0}^k p^{-r/2}\int_{W} b_r(x)dv_Y(x)+O(p^{\frac{l}{2}-\frac{k+1}{2}})\\
& =\frac{m_{X,\Lambda_2}}{m_{\Lambda_1,Y}}p^{\frac{l}{2}}\lambda^p\sum_{r=0}^k p^{-r/2}\int_{W} b_r(x)dv_Y(x)+O(p^{\frac{l}{2}-\frac{k+1}{2}}).
\end{split}
\end{equation}

We can then go on to the proof of \cref{theointergal} to get \cref{orbibm0} in the case $F=\Id_E$. Now for the general case, if $F\in\cinf(X,\End(E))$ has compact support, we can define its Berezin-Toeplitz quantization by \cref{Toep}, and it is showed in \cite[Lemma~6.10]{MM08b} that it satisfies \cref{Toepasy} as well. Furthermore, the formula \cref{Toeplag} holds in the same way.

Finally, let us consider the case of singular Bohr-Sommerfeld submanifolds. Following \cref{orbirep}-\cref{orbiprodcomput}, it suffices to prove the case $m_{X,Y}=1$, and as we assumed the intersection locus away from the singular set, we need only to prove the analogue of \cref{norme}, and suppose that $f$ has compact support in some $U\in\cal{U}_X$. 

First recall that the reproducing property gives
\begin{multline}\label{singrep}
\norm{s_{f,p}}_p^2 =\int_\Lambda\<s_{f,p}(\iota(x)),\iota_p.\zeta^pf(x)\>_{E_p}dv_\Lambda(x)\\
=\int_{U}\int_{U} \<\til{P}_p(\til{\iota}_V(x),\til{\iota}_V(y))\til{\iota}_p.\til{\zeta}^p\til{f}(y),\til{\iota}_p.\til{\zeta}^p\til{f}(x)\>_{E_p}dv_\Lambda(y)dv_\Lambda(x)\\
=\sum_{g\in G_V}\int_{U}\int_{U} \<F(\til{D}_p)(\til{\iota}_V(x),g\til{\iota}_V(y))g.\til{\iota}_p.\til{\zeta}^p\til{f}(y),\til{\iota}_p.\til{\zeta}^p\til{f}(x)\>_{E_p}dv_\Lambda(y)dv_\Lambda(x)\\
+O(p^{-\infty}).
\end{multline}
Now, as $G_V$ acts on $\til{V}$ preserving all the structures and by \cref{orbimap}, the immersion $g\til{\iota}_V$ is an isotropic immersion intersecting $\til{\iota}_V$ cleanly, for any $g\in G_V$. As $F(\til{D}_p)(\cdot,\cdot)$ satisfies the expansion of \cref{asy}, we can then apply \cref{theointergal} to compute each term of the last line of \cref{singrep}. We then have an asymptotic expansion of the form \cref{Toepb0norme}.

To compute the first order term, note that if $g\til{\iota}_V$ and $\til{\iota}_V$ do not coincide, the highest order of the corresponding expansion \cref{Toepb0norme} is strictly smaller than $n/2$. Thus we need only to consider the subgroup of $G_V$ fixing the image of $\iota$, which contains at least the identity element of $G_V$. Summing the contributions of all the elements of this subgroup and by \cref{bm0}, we get a function $b_U\in\cinf(U,\C)$, depending on $f$ only locally, such that the highest order term of \cref{singrep} is given by integration of $b_U$ along $U$. Now, as $\iota^{-1}(X_{sing})$ is of measure $0$, we can pick a sequence $U_n\subset U,~n\in\N$, of open sets in $\cal{U}_\Lambda$ containing $\iota^{-1}(X_{sing})$ and whose measure tends to $0$. We can then repeat \cref{singrep} replacing $U$ by $U_n$ and use \cref{orbib0norme} on the regular part of $V$ to get the following formula for the highest order term, for all $n\in\N$,
\begin{equation}\label{bm0sing}
b_0=2^{d/2}\int_{\Lambda\backslash U_n}\<Ff(x),f(x)\>_{\iota^*E}\,dv_{\Lambda}(x)+\int_{U_n}b_U(x)\,dv_\Lambda(x).
\end{equation}
As the second term can be made arbitrarily small, we can take the limit of \cref{bm0sing} at $n$ tends to infinity, so that formula \cref{bm0} holds for singular Bohr-Sommerfeld submanifolds.
\end{proof}

\section{Application to relative Poincaré series}\label{appli}

In this Section, we apply the results of the previous section
in the case of quotients of the hyperbolic plane $\IH$
by a discrete subgroup $\Gamma$ of $\text{SL}_2(\R)$.
In that case, the Bergman kernel admits
an explicit global formula given in \cref{BergH} as a
sum over $\Gamma$, realizing it as a \emph{Poincaré series}.
In \cref{RPSgal}, we show that the isotropic states associated
with remarkable curves over $\IH/\Gamma$ can then
be expressed as \emph{relative Poincaré
series}, where the sum is over a quotient of $\Gamma$ instead.
The main result of this section is \cref{thBScurve}, which is an
explicit version of \cref{introtheonorme} in this setting,
and which shows that such relative Poincaré series do not
vanish as soon as their weight as \emph{holomorphic
cusp forms} is large enough.

Recall that the special linear group 
\begin{equation}
\text{SL}_2(\R)=\bigg\{g=
\begin{pmatrix} 
a & b \\
c & d 
\end{pmatrix}
~\bigg|~a,b,c,d\in\R,~ad-bc=1 \bigg\}
\end{equation}
acts on the Poincaré upper-half plane $\IH=\left\{z=x+\sqrt{-1}y\in\C~\big|~y>0\right\}$ by the formula
\begin{equation}
g.z=\frac{az+b}{cz+d}.
\end{equation}
The induced action of $g$ on the canonical holomorphic vector field
$\partial/\partial z$ over $\IH$ is given by $g.\partial/\partial z
=(cz+d)^{-2}\partial/\partial z$, so that the dual action on the
canonical line bundle
$K_\IH=T^{*(1,0)}\IH$ over $\IH$ is given on the canonical
section $dz$ by
\begin{equation}\label{actcan}
g.dz=(cz+d)^{2}dz=:j(g,z)^{2}dz.
\end{equation}
Let $g^{T\IH}$ be the \emph{hyperbolic metric} on $\IH$, 
which is defined by the formula
\begin{equation}
g^{T\IH}=\frac{dx^2+dy^2}{y^2},
\end{equation}
so that it is invariant by the action of $\text{SL}_2(\R)$.
The associated Kähler metric $\om_\IH$ satisfies
\begin{equation}
\om_\IH=\frac{\sqrt{-1}}{2}\frac{dz\wedge d\bar{z}}{y^2}.
\end{equation}
Let us write $|\cdot|_{K_\IH}$ for the $\text{SL}_2(\R)$-invariant
Hermitian norm on $K_\IH$ given by
\begin{equation}\label{metKH}
|dz|_{K_\IH}^2=y^2.
\end{equation}
Note that it differs from the norm induced by $g^{T\IH}$ from
a constant factor $\sqrt{2}$.
Then the curvature $R^{K_{\IH}}$ of the Chern connection of
$(K_{\IH},h^{K_{\IH}})$ satisfies $\sqrt{-1}R^{K_{\IH}}=\om_{\IH}$,
so that $R^{K_{\IH}}$ satisfies the condition \cref{preq} for the
renormalized Kähler form $\om_\IH/2\pi$. As $R^{det}=-R^{K_\IH}$
is proportional to $\sqrt{-1}\om_\IH$,
it is easily seen that $K_\IH$ satisfies \cref{hypL2}.

Now if $\Gamma$ is a discrete subgroup of $\text{SL}_2(\R)$, the quotient $X:=\IH/\Gamma$ has an induced structure of a Kähler orbifold, and its canonical line bundle $K_X$ is the quotient of $K_\IH$ by the induced action \cref{actcan}. We denote $g^{TX}$ and $\om_X$ for the quotient metric and quotient Kähler form on $X$ respectively,
and we endow $K_X$ with the Hermitian metric $h^{K_X}$ induced
by \cref{metKH}.
Then \cref{preq} holds for $K_X$ up to a factor $2\pi$ as
above, and it satisfies \cref{hypL2} as well. Therefore,
setting $L=K_X$ and $E=\C$, we are precisely in the context of the
previous sections for the renormalized Kähler form
$\om=\om_{X}/2\pi$, with $g^{TX}_\om=g^{TX}/2\pi$.

Recall that a smooth path $\gamma:[0,l]\fl X,~l>0,$ is said to be a \emph{closed loop} if it induces a (singular) immersion $\til{\gamma}:S^1\fl X$ by identification of $0$ with $l$. The following lemma describes the class of (singular) Bohr-Sommerfeld submanifolds we will be interested in.

\begin{lem}\label{BScurve}
For $l>0$, let $\gamma:[0,l]\fl X$ be a closed loop in $X$ parametrized by arclength with respect to $g^{TX}$, and suppose that the holonomy of $K_X$ along $\gamma$ with respect to $\nabla^{K_\IH}$ is trivial. Then the immersion $\til{\gamma}:S^1\fl X$, obtained from $\gamma$ by identification of $0$ and $l$, satisfies the Bohr-Sommerfeld condition of \cref{BS}.
\end{lem}
\begin{proof}
As $\om_X$ is a $2$-form, any smooth map $f:S^1\fl X$ satisfies $f^*\om=0$. Thus as $\dim X=2$, any immersion $\iota:S^1\fl X$ is Lagrangian. By \cref{hol}, it satisfies the Bohr-Sommerfeld condition if and only if the holonomy of the pullback connection is trivial, which is exactly the hypothesis of \cref{BScurve} by \cref{hol}.
\end{proof}

In any case, such a path $\gamma:[0,l]\fl X,~l>0,$ is called a \emph{Bohr-Sommerfeld curve}. The orientation on $\til{\gamma}:S^1\fl X$ is determined by the canonical vector field $\partial_t$
on $[0,l]$. Following \cref{hol}, if $\gamma:[0,l]\fl X,~l>0,$ is a smooth closed loop such that its holonomy is a $k$-th root of unity for some $k\in\N$, we can take a cover of degree $k$ of this loop to get a Bohr-Sommerfeld curve $\gamma_k:[0,kl]\fl X$.

Note that as $X$ is a complex orbifold with $\dim_\C X=1$ and as $\Gamma$ acts on $\IH$ holomorphically, the singular set $X_{sing}$ is necessarily a discrete set. By \cref{orbimap} and as $S^1$ is a manifold, the stabilizer of $\til{\gamma}$ is then necessarily trivial in any case.

\begin{cor}\label{BSgeod}
A closed geodesic loop $\gamma:[0,l]\fl X,~l>0$, parametrized by arclength, is a Bohr-Sommerfeld curve.
\end{cor}
\begin{proof}
Recall that $K_X=T^{*(1,0)}X$ is equipped with the Hermitian metric and connection $h^{K_X},\nabla^{K_X}$ induced by $g^{TX},\nabla^{TX}$ via \cref{splitc}. For any $t\in[0,l]$, let $\dot\gamma_t\in T_{\gamma(t)}X$ denote the vector tangent to the curve $\gamma:[0,l]\fl X$, inducing $\dot\gamma^{(0,1)}_t\in T^{(0,1)}X$ via \cref{splitc}. We write $\dot\gamma_t^{(0,1),*}\in K_{X,\gamma(t)}$ for its metric dual. As $\gamma:[0,l]\fl X$ is geodesic, we know that $\nabla^{TX}_{\dot\gamma}\dot\gamma=0$, so that $\nabla^{K_X}_{\dot\gamma}\dot\gamma^{(0,1),*}=0$, which means precisely that $\til{\gamma}:S^1\fl X$ satisfies the Bohr-Sommerfeld condition with associated section $\gamma^{(0,1),*}\in\cinf(S^1,\til{\gamma}^*K_X)$.

Now if $X$ is an orbifold and if $z\in X$ is a singular point of $X$, then its associated group $G_z^X$ preserves the Riemannian structure, and sends a geodesic through $z$ to another geodesic through $z$, which intersect transversally by unicity of the geodesics. Thus $\gamma:[0,l]\fl X$ satisfies the definition of a singular immersion as in \cref{orbimap}.
\end{proof}

Let $\gamma:[0,l]\fl X,~l>0,$ be a Bohr-Sommerfeld curve together with a unitary flat section $\zeta\in\cinf([0,l],\gamma^* K_X)$, inducing a (possibly singular) Bohr-Sommerfeld submanifold $(S^1,\til{\gamma},\zeta)$ as above. For any $p\in\N^*$, we define $s_{\gamma,p}\in H^0_{(2)}(X,K_X^p)$ by
\begin{equation}\label{BScurvestate}
s_{\gamma,p}(x)=\int_0^l P_p^X(x,\gamma(t))\gamma_p.\zeta^p(t) dt,
\end{equation}
for any $x\in X$, where $P_p^X(\cdot,\cdot)$ is the Bergman kernel with respect to $dv_X$ of the orthogonal projection on $H^0_{(2)}(X,K_X^p)$. Then $s_{\gamma,p}$ is precisely the Lagrangian state associated with $(S^1,\til{\gamma},\zeta)$ and $f=1$, in the sense of \cref{Lagstate}.

We can then apply \cref{thL2} and \cref{orbith} to get the following specialisation of \cref{b0norme} and \cref{bmj0}, where we adopt the convention that $\sqrt{-a}=\sqrt{-1}\sqrt{a}$ if $a>0$.

\begin{theorem}\label{thBScurve}
Let $\gamma:[0,l]\fl X,~l>0,$ be a Bohr-Sommerfeld curve, and let $\{s_{\gamma,p}\}_{p\in\N^*}$ be as in \cref{BScurvestate}. Then
\begin{equation}\label{BScurvenorme}
\norm{s_{\gamma,p}}_{L^2}^2=\left(\frac{p}{\pi}\right)^{1/2}l+O(p^{-1/2}).
\end{equation}

Furthermore, if $\gamma_1$ and $\gamma_2$ are two Bohr-Sommerfeld curves intersecting cleanly away from the singular set, we get
\begin{equation}\label{BScurveinter}
\<s_{\gamma_1,p},s_{\gamma_2,p}\>=\sqrt{2}\sum_{z\in\gamma_1\cap\gamma_2}\sum_{\substack{t_1,t_2>0,\\
\gamma_1(t_1)=\gamma_2(t_2)=z}}\lambda^p_{t_1,t_2}\frac{e^{\sqrt{-1}(\theta_z/2-\pi/4)}}{\sqrt{\sin(\theta_z)}}+O(p^{-1}),
\end{equation}
where $\theta_z\in\ ]0,2\pi[$ is the oriented angle from $\gamma_1$ to $\gamma_2$ at $z$ and where for all $t_1, t_2>0$ such that $\gamma_1(t_1)=\gamma_2(t_2)$, we define $\lambda_{t_1,t_2}=\<\zeta_1(t_1),\zeta_2(t_2)\>_{K_X}$.
\end{theorem}
\begin{proof}
In the case $X$ smooth and compact, \cref{BScurvenorme,BScurveinter} are standard computations from \cref{b0,bmj0}. We will indicate how to modify directly the argument to get the case $g^{TX}=2\pi g^{TX}_\om$ from the case $g^{TX}=g^{TX}_\om$ in all generality.

For any $p\in\N^*$, let us write $P_{p,\om}$ for the orthogonal projection to $H^0_{(2)}(X,K_X^p)$ with respect to the $L^2$-Hermitian product induced by $g^{TX}_\om$. Then $P_{p,\om}=P_p^X$, but $dv_{X,\om}=dv_X/2\pi$, so that the associated Bergman kernel with respect to $dv_{X,\om}$ satisfies $P_{p,\om}(\cdot,\cdot)=2\pi P_p^X(\cdot,\cdot)$. On the other hand, the Riemannian volume form $dt_\om$ on $[0,L[$ induced by $g^{TX}_\om$ satisfies $dt_\om=dt/\sqrt{2\pi}$. Thus, writing $\{s_{\om,\gamma,p}\}_{p\in\N^*}$ for the Lagrangian state obtained replacing $g^{TX}$ by $g^{TX}_\om$, we get from \cref{defLagstate} that $s_{\om,\gamma,p}=\sqrt{2\pi}s_{\gamma,p}$ for any $p\in\N^*$. 

Consider now two Bohr-Sommerfeld curves $\gamma_1$ and $\gamma_2$. Following the above notations, we get for any $p\in\N^*$,
\begin{equation}\label{uv}
\<s_{\gamma_1,p},s_{\gamma_2,p}\>_p =\frac{1}{2\pi}\int_X \<s_{\om,\gamma_1,p},s_{\om,\gamma_2,p}\>_{K_X^p} dv_X=\<s_{\om,\gamma_1,p},s_{\om,\gamma_2,p}\>_{\om,p},
\end{equation}
where $\<\cdot,\cdot\>_{\om,p}$ denote the $L^2$-Hermitian product with respect to $g^{TX}_\om$. Noticing finally that $\Vol_\om(\gamma)=l/\sqrt{2\pi}$ for any $\gamma:[0,l]\fl X,~l>0$ parametrized by arclength with respect to $g^{TX}$, we recover \cref{BScurvenorme,BScurveinter} as in the case of $X$ smooth and compact.
\end{proof}

In the case where $X$ is a compact Riemann surface, so that in
particular $\Gamma$ acts freely on $\IH$, \cref{thBScurve} is the
result of \cite[Th.~4.4]{BPU95}, where they show \cref{BScurvenorme}
and \cref{BScurveinter} with a weaker error term. As shown in
\cref{RPSgal}, formulas \cref{BScurvenorme} and \cref{BScurveinter} 
are especially interesting in the case of curves $\gamma:\R\fl\IH$ 
such that there exists $l>0,~g_0\in\Gamma$ satisfying $g_0.\gamma(t)=
\gamma(t+l)$ for any $t\in\R$. We say that $\gamma$ is 
\emph{associated with} $g_0$.

In particular, if $\gamma$ is a closed geodesic, then $\gamma$ is associated with an \emph{hyperbolic} element $g_0\in\Gamma$, that is satisfying $\Tr(g_0)>2$, unique up to conjugation. Closed geodesics belong to a larger class of hyperbolic curves called \emph{hypercycles}.

If $g_0\in\Gamma$ is \emph{parabolic}, that is satisfying $\Tr(g_0)=2$, then its action has no fixed points in $\IH$, and it occurs in $\Gamma$ only in the case of $X$ non-compact. The most interesting associated curves in that case are the so-called \emph{horocycles}, which are isometric to a horizontal line in $\IH$.

If $g_0\in\Gamma$ is \emph{elliptic}, that is satisfying $\Tr(g_0)<2$, then $g_0$ fixes a unique point $z\in\IH$, which descends to a singular point of $X$. The most interesting associated curves in that case are circles with center the fixed point of $g_0$ in $\IH$. Note that $\Gamma$ acts freely on $\IH$ if and only if it contains no elliptic elements.

%

Our next goal is to identify explicitly the Lagrangian states
associated with such curves. Let $\FF$ be a measurable fundamental
domain of $\Gamma$ in $\IH$. Through the natural identification
$\cinf(X,K_X)\simeq\cinf(\IH,K_\IH)^\Gamma$ and trivializing
$K_\IH$ using its canonical section $dz$, we have from
\cref{actcan} and for any $p\in\N^*$ the following natural
identification,
\begin{multline}\label{S2p}
H^0_{(2)}(X,K_X^p)\simeq\big\{f\in\cinf(\IH)~\big|~f\ \text{holomorphic},\\
f(g.z)=f(z)j(g,z)^{2p},\ \int_{\FF} |f(z)|^2y^{2p-2}dxdy<\infty\big\}.
\end{multline}
This identification will be used implicitly throughout the rest of this section.

\begin{rem}\label{cuspform}
Assume $\Vol(X)<+\infty$, that is $\Gamma$ is a \emph{Fuchsian group of the first kind}. As explained in \cite[\S\ 6]{AMM16a,AMM16},
the space $H^0_{(2)}(X,K_X^p)$ is then identified through the identification \cref{S2p} with the space $\text{S}_{2p}(\Gamma)$ of
\emph{holomorphic cusp forms of weight $2p$}, the space of holomorphic functions on $\IH$ satisfying the equivariance property of \cref{S2p} and vanishing at infinity. Such spaces are of particular interest in arithmetic.
\end{rem}

The following result is classical and follows for instance from \cite[Prop.~I.5.3, II.1]{Frei90}.

\begin{prop}\label{BergH}
Under the identifications above, for any $p\in\N^*$,
the Bergman kernel of $H^0_{(2)}(\IH,K_\IH^p)$ satisfies the formula
\begin{equation}\label{fleBergH}
P_p^\IH(z,w)=\frac{2p-1}{4\pi}\left(\frac{2\sqrt{-1}}
{z-\overline{w}}\right)^{2p}dz^p d\overline{w}^p,
\end{equation}
for any $z,w\in\IH$, where
$d\overline{w}\in \overbar{K}_{\IH,w}\simeq K^*_{\IH,w}$ denotes the
metric dual of $dw\in K_{\IH,w}$. Furthermore, for any
$\til{w}\in\IH$ descending to $w\in X$ in the quotient, we have
\begin{equation}\label{berginv}
P_p^X(z,w)=\sum_{g\in\Gamma} j(g,z)^{-2p}P^\IH_p(g.z,\til{w}),
\end{equation}
through the identification \cref{S2p} in $z\in\IH$,
where the convergence of the right-hand side is absolute and
uniform for $z,\,\til{w}$ in any compact set of $\IH$.
\end{prop}

%
%

The series \cref{berginv} is an example of \emph{Poincaré series}, and is a standard method to construct functions in $\text{S}_{2p}(\Gamma)$ as in \cref{cuspform}. A fundamental problem of the theory of cusp forms is to decide whether a given series vanishes identically or not.

If $\Gamma_0\subset\Gamma$ is a subgroup of $\Gamma$, let us write $\Gamma_0\backslash\Gamma$ for the set of equivalence classes $[g]:=\{g_0g\in\Gamma\ |\ g_0\in\Gamma_0\}$ for all $g\in\Gamma$. Recall that if $g_0$ is hyperbolic or parabolic, it generates a free group $\Gamma_0\subset\Gamma$, whereas if $g_0$ is elliptic, it generates a cyclic subgroup $\Gamma_0\subset\Gamma$.

Using \cref{BergH} and a classical unfolding technique, we get explicit formulas for the Lagrangian states associated with remarkable curves. This is described in the next result.
\begin{prop}\label{RPSgal}
Let $g_0\in\Gamma$, and let $\gamma:\R\fl\IH$ be a smooth curve on $\IH$ parametrized by arclength, together with a unitary flat section $\zeta \in \gamma^*K_\IH$, such that there is an $l>0$ satisfying $g_0.\gamma(t)=\gamma(t+l)$ and $g_0.\zeta(t)=\zeta(t+l)$ for all $t\in\R$. Write $\Gamma_0\subset\Gamma$ for the subgroup generated by $g_0$.

If $g_0$ is hyperbolic or parabolic, then the Lagrangian state $\{s_{\gamma,p}\}_{p\in\N^*}$ associated with $\gamma$ is given through \cref{S2p} and for any $p\in\N^*$ by
\begin{equation}\label{RPS}
\begin{split}
s_{\gamma,p}(z)=\frac{2p-1}{4\pi}
\sum_{[g]\in\Gamma_0\backslash\Gamma}j(g,z)^{-2p}
\int_{-\infty}^{+\infty}\left(\frac{2\sqrt{-1}}
{g.z-\overbar{\gamma(t)}}\right)^{2p}
\big\langle\zeta(t), d\gamma(t)\big\rangle_{K_X}dt.
\end{split}
\end{equation}
If $g_0$ is elliptic, then letting $n\in\N$ be the order of $\Gamma_0$, the Lagrangian state $\{s_{\gamma,p}\}_{p\in\N^*}$ is given through \cref{S2p} and for any $p\in\N^*$ by
\begin{equation}\label{RPSell}
\begin{split}
s_{\gamma,p}(z)=\frac{2p-1}{4\pi}
\sum_{[g]\in\Gamma_0\backslash\Gamma}j(g,z)^{-2p}
\int_0^n\left(\frac{2\sqrt{-1}}{g.z-\overbar{\gamma(t)}}\right)^{2p}\big\langle \zeta(t),d\gamma(t)\big\rangle_{K_X}dt.
\end{split}
\end{equation}
The convergence of the series in \cref{RPS} and \cref{RPSell} are absolute and uniform in $z$ in any compact set of $\IH$.
\end{prop}
\begin{proof}
Recall that $\text{SL}_2(\R)$ acts on $\IH$ by holomorphic isometries
and that the induced action on $K_X$ preserves $h^{K_X}$.
This implies in particular that the Bergman kernel of
$H^0_{(2)}(X,K_X)$ is invariant by $\text{SL}_2(\R)$.
Using \cref{actcan}, for any
$w\in\IH,\,g\in \text{SL}_2(\R)$ and $\zeta\in K_{\IH,w}$, we have
\begin{equation}
j(g,z)^{-2p}P^{\IH}_p(g.z,w)\zeta=P^{\IH}_p(z,g^{-1}.w)g^{-1}.\zeta,
\end{equation}
through the identification \cref{S2p} in $z\in\IH$.
On the other hand, for any $g,h\in \text{SL}_2(\R)$ and $w\in\IH$,
the cocycle formula $j(gh,w)=j(g,h.w)j(h,w)$ holds by definition.
Consider $g_0\in\Gamma$ hyperbolic or parabolic, and let
$l>0$ be the smallest positive number
satisfying $g_0.\gamma(t)=\gamma(t+l)$ and $g_0.\zeta(t)=\zeta(t+l)$
for all $t\in\R$.
Then from \cref{BScurvestate} and from the uniform convergence
of \cref{berginv}, we get
\begin{equation}\label{sumsPStoLS}
\begin{split}
s_{\gamma,p}(z) & =\int_{\gamma}\sum_{g\in\Gamma}
j(g,z)^{-2p}
P_p^\IH(g.z,\gamma(t))\zeta(t)dt\\
& =\sum_{[g]\in\Gamma_0\backslash\Gamma}~\sum_{n\in\Z}
j(g_0^ng,z)^{-2p}
\int_0^{l}P_p^\IH(g_0^ng.z,\gamma(t))\zeta(t)dt\\
& =\sum_{[g]\in\Gamma_0\backslash\Gamma}j(g,z)^{-2p}
\sum_{n\in\Z}
\int_0^{l}P_p^\IH(g.z,g_0^{-n}.\gamma(t))
g_0^{-n}.\zeta(t)dt\\
& =\sum_{[g]\in\Gamma_0\backslash\Gamma}j(g,z)^{-2p}
\sum_{n\in\Z}\int_{-nl}^{-(n+1)l}P_p^\IH(g.z,\gamma(t))\zeta(t)dt\\
& =\sum_{[g]\in\Gamma_0\backslash\Gamma}j(g,z)^{-2p}
\int_{-\infty}^{+\infty}P_p^\IH(g.z,\gamma(t))\zeta(t)dt,
\end{split}
\end{equation}
and we conclude by \cref{fleBergH}. Note that the sums in \cref{sumsPStoLS} do not depend on the choice of the representatives $g\in\Gamma$ of any $[g]\in\Gamma_0\backslash\Gamma$. The elliptic case \cref{RPSell} is strictly analogous.
\end{proof}

The series \cref{RPS} and \cref{RPSell} are called \emph{relative Poincaré series}. We can now state our main theorem, which is a consequence of \cref{thBScurve}.

\begin{theorem}\label{maincor}
If $\gamma:\R\fl\IH$ satisfying the hypotheses of \cref{RPSgal} descends to a Bohr-Sommerfeld curve, then there is a $p_0\in\N$ such that the associated series \cref{RPS} or \cref{RPSell} do not vanish identically for $p>p_0$. This holds in particular if $\gamma:\R\fl\IH$ is a closed geodesic.
\end{theorem}
\begin{proof}
By \cref{BScurvenorme}, we know that there is $p_0\in\N$ such that $s_{\gamma,p}$ is non-vanishing for $p\geq p_0$, so that we may conclude by \cref{BSgeod} and \cref{RPSgal}.
\end{proof}

In general, there are simple numerical criterions for horocycles, circles and hypercycles to satisfy the Bohr-Sommerfeld condition, and the integral in the sums \cref{RPS} and \cref{RPSell} can be computed explicitly using \cref{BergH} and elementary complex analysis. In particular, as computed in \cite[Th.~4.11]{BPU95}, if $g_0=
(\begin{smallmatrix} 
a & b \\
c & d 
\end{smallmatrix})$ is a hyperbolic element of $\Gamma$, the series \cref{RPS} for $\gamma$ closed geodesic associated with $g_0$ takes the form
\begin{equation}\label{RPShyp}
s_{\gamma,p}(z)=C_p\sum_{[g]\in\Gamma_0\backslash\Gamma}j(g,z)^{-2p}
\left(c(g.z)^2+(d-a)(g.z)-b\right)^{-p},
\end{equation}
with $C_p\in\C$ explicit non-vanishing constant for all $p\in\N^*$,
and we recover up to normalisation the relative Poincaré series associated with closed hyperbolic geodesics by Katok \cite[\S\ 1]{Kat85}. Furthermore, we get from \cref{thBScurve} a formula for the highest order term as $p\fl+\infty$ of the intersection product of two closed geodesics, recovering a result of \cite[Th.~3]{Kat85}. As showed in \cite[Th.~1]{Kat85}, if $\Gamma$ is a Fuchsian group of the first kind, the series associated with the primitive hyperbolic elements of $\Gamma$ as above generate the whole space $\text{S}_{2p}(\Gamma)$.
%

Finally, note that there are many discrete subgroups $\Gamma\subset \text{SL}_2(\R)$ of interest containing elliptic points and leading to non-compact quotients $X=\IH/\Gamma$, even in the case when
$\Gamma$ is Fuchsian of the first kind. The most famous examples are the classical modular curves.

\bibliographystyle{amsplain}

\providecommand{\bysame}{\leavevmode\hbox to3em{\hrulefill}\thinspace}
\providecommand{\MR}{\relax\ifhmode\unskip\space\fi MR }
\providecommand{\MRhref}[2]{%
  \href{http://www.ams.org/mathscinet-getitem?mr=#1}{#2}
}
\providecommand{\href}[2]{#2}

\ \\

\ \\
Tel Aviv University - School of Mathematical Sciences,\\
Ramat Aviv, Tel Aviv 69978, Israël\\
\\
\emph{E-mail adress}: louisioos@mail.tau.ac.il

\end{document}